\documentclass{amsart}          
\usepackage[all]{xy}

\usepackage{hyperref}

\usepackage{cleveref}

\usepackage{amssymb}

\usepackage{xcolor}

\usepackage{chngcntr} \usepackage{apptools}
\AtAppendix{\counterwithin{theorem}{subsection}}

\usepackage{stmaryrd}

\usepackage{relsize} \usepackage[bbgreekl]{mathbbol}
\usepackage{amsfonts} \usepackage{todonotes}

\DeclareSymbolFontAlphabet{\mathbb}{AMSb}
\DeclareSymbolFontAlphabet{\mathbbl}{bbold}

\theoremstyle{plain} 
\newtheorem{theorem}{Theorem}[section]
\newtheorem{proposition}[theorem]{Proposition}
\newtheorem{lemma}[theorem]{Lemma}
\newtheorem{corollary}[theorem]{Corollary}
\newtheorem{conjecture}[theorem]{Conjecture}
 \theoremstyle{definition}
\newtheorem{definition}[theorem]{Definition}
\newtheorem{example}[theorem]{Example}
\newtheorem{remark}[theorem]{Remark}

 \newcommand{\Aut}{\mathrm{Aut}}
\newcommand{\Av}{\mathrm{Av}} \newcommand{\Bun}{\mathrm{Bun}}
\newcommand{\bs}{\blacksquare} 
 
 \newcommand{\CT}{\mathrm{CT}}
\newcommand{\Div}{\mathrm{Div}} 
\newcommand{\Eis}{\mathrm{Eis}}
\newcommand{\et}{\mathrm{{\acute{e}t}}} \newcommand{\F}{\mathbb{F}}

\newcommand{\Flbar}{\overline{\mathbb{F}}_\ell}
 \newcommand{\GL}{\mathrm{GL}}
\newcommand{\Gm}{{\mathbb{G}_m}} 
\newcommand{\Hom}{\mathrm{Hom}} 
 
 \newcommand{\lis}{\textrm{lis}}
\newcommand{\Mod}{\mathrm{Mod}} 
\newcommand{\nat}{\natural} \newcommand{\Perf}{\mathrm{Perf}}
 \newcommand{\Q}{\mathbb{Q}}
\newcommand{\Qlbar}{\overline{\Q}_\ell}
\newcommand{\Zlbar}{\overline{\Z}_\ell}
 
 \newcommand{\Rep}{\mathrm{Rep}}
 
 \newcommand{\Spec}{\mathrm{Spec}}

 \newcommand{\Z}{\mathbb{Z}}

\begin{document}

\title[Averaging functors in Fargues' program for $GL_n$]{Averaging
  functors in Fargues' program for $\GL_n$}

\author{Johannes Ansch\"{u}tz, Arthur-C\'esar Le Bras}
\email{ja@math.uni-bonn.de} \email{lebras@math.uni-bonn.de}
\date{\today}

\begin{abstract}
 We study the so-called averaging functors from the geometric Langlands program in the setting of Fargues' program. This makes explicit certain cases of the spectral action which was recently introduced by Fargues-Scholze in the local Langlands program for $\mathrm{GL}_n$. Using these averaging functors, we verify (without using local Langlands) that the Fargues-Scholze parameters associated to supercuspidal modular representations of $\mathrm{GL}_2$ are irreducible. We also attach to any irreducible $\ell$-adic Weil representation of degree $n$ an Hecke eigensheaf on $\Bun_n$, and show, using the local Langlands correspondence and recent results of Hansen and Hansen-Kaletha-Weinstein, that it satisfies most of the requirements of Fargues' conjecture for $\GL_n$.  
\end{abstract}

\maketitle

\tableofcontents

\section{Introduction}
\label{sec:introduction}
Let $p$ be a prime and fix a non-archimedean local field $E$ of residue characteristic $p$.
Fix an algebraic closure $\overline{E}$ of $E$ with associated Weil group $W_E$, and a prime $\ell\neq p$.
In \cite{fargues_geometrization_of_the_local_langlands_correspondence_overview}, Fargues formulated his program to geometrize the local Langlands correspondence for $E$ via the stack $\Bun_G$ of $G$-torsors on the Fargues-Fontaine curve, with $G$ a reductive group over $E$. Roughly, he suggested that the geometric Langlands program as initiated by Drinfeld \cite{drinfeld_two_dimensional_l_adic_representations} and Laumon \cite{laumon_duke} can be suitably extended to the situation where one replaces the usual smooth projective curve over a finite field by the Fargues-Fontaine curve, and that this should yield results on the local Langlands program for $E$.

Fargues' program has been put forward in the seminal recent work of Fargues-Scholze \cite{fargues2021geometrization}. One of the main theorems of \cite{fargues2021geometrization} is the existence of the spectral action of the (Ind)-category of perfect complexes on the stack of $L$-parameters for a reductive group $G$ over $E$ on the category
\[
  D_{\mathrm{lis}}(\Bun_G,L)
\]
of ``lisse'' sheaves with coefficients in an arbitrary $\Z_\ell$-algebra $L$\footnote{Depending on $G$ one has to put some assumption on $\ell$. However, for $G=\GL_n$ each $\ell$ works.}.
Assume from now on that $L$ is a field, e.g., $L=\Flbar$ or $L=\Qlbar$. When writing ``representation'', we implicitly mean ``representation over $L$'' in the following.

The existence of the spectral action naturally implies that one can construct an $L$-parameter $\varphi_\pi\colon W_E\to {}^LG$ for any irreducible, smooth $G(E)$-representation $\pi$. For $G=\GL_n$ and $\pi$ irreducible, supercuspidal, the knowledge of the cohomology of the Lubin-Tate spaces implies that $\varphi_\pi$ agrees with the previously constructed local Langlands correspondent of $\pi$ (cf. \cite[Theorem 9.7.4]{fargues2021geometrization}). In particular, $\varphi_\pi$ is irreducible in this case.
Our first main result provides for $n=2$ an alternative proof of this fact.

\begin{theorem}[{cf.\ \Cref{sec:appl-irred-farg-on-cusp-hecke-action-contains-no-rank-1-stuff}}]
  \label{sec:introduction-1-main-theorem-on-irreducibility}
  Let $\pi$ be an irreducible supercuspidal representation of $\GL_2(E)$ over an algebraically closed field $L$ of characteristic $\ell$. Then the Fargues-Scholze parameter
  \[
    \varphi_\pi\colon W_E\to \GL_2(L)
  \]
  associated with $\pi$ is irreducible.
\end{theorem}

The assumption that $L$ has characteristic $\ell$ comes from foundational problems in the theory of ``lisse'' sheaves (more specifically, problems related to excision). Though the result of \Cref{sec:introduction-1-main-theorem-on-irreducibility} is not so interesting in its own right, we think that its strategy of proof is.
Namely, we completely bypass the knowledge of the cohomology of the Lubin-Tate tower (and thus also any global methods).
To implement our strategy we first have to reinterpret the condition of being supercuspidal geometrically. Let $L$ be algebraically closed of characteristic $\ell$, as in the theorem. We do this using the geometric constant term functors, which are familiar from the geometric Langlands program, and prove that a smooth irreducible representation of $\GL_2(E)$ is supercuspidal (in the sense of modular representation theory) if and only if its corresponding object in $D_\et(\Bun_2,L)$ is a cuspidal object (in the sense of the geometric Langlands program), cf.\ \Cref{sec:geom-cusp-repr-geom-cusp-representations-for-n-2}. It's natural to expect that such a statement is true for $n\geq 3$ as well. Let
\[
  D_{\et,\mathrm{cusp}}(\Bun_2,L)\subseteq D_\et(\Bun_2,L)
\]
be the full subcategory of cuspidal objects. The next step is to analyze how the averaging functor
\[
  \mathrm{Av}_{\mathbb{L}}\colon D_\et(\Bun_2,L)\to D_\et(\Bun_2,L)
\]
associated with some $W_E$-representation $\mathbb{L}$ acts on $D_{\et,\mathrm{cusp}}(\Bun_2,L)$. Averaging functors have been introduced in the geometric Langlands program by Frenkel-Gaitsgory-Vilonen \cite{frenkel_gaitsgory_vilonen} and they have been further studied in \cite{gaitsgory_on_a_vanishing_conjecture}. Translating some known geometric arguments about them to the (easier case of the) Fargues-Fontaine curve, we can prove that
\[
  \mathrm{Av}_{\mathbb{L}}
\]
vanishes on $D_{\et,\mathrm{cusp}}(\Bun_2,L)$ if $\mathbb{L}$ is irreducible of dimension $\neq 2$. From here we can then conclude \Cref{sec:introduction-1-main-theorem-on-irreducibility}.

The averaging functors feature also prominently in the proof of our second main result, where we may assume that $n\geq 1$ is arbitrary. Before stating it, let us recall that for each $b \in B(\GL_n)$ basic, the substack $j_b: \Bun_n^b \hookrightarrow \Bun_n$ parametrizing vector bundles fiberwise on the base isomorphic to $\mathcal{E}_b$ is open and isomorphic to classifying stack
$$
[\ast/\underline{G_b(E)} ],
$$ 
where $G_b$ denotes the $\sigma$-centralizer of $b$ (an inner form of $\GL_n$), inducing an equivalence
$$
D(\mathrm{Rep}_L^\infty(G_b(E))) \overset{\sim}\to D_\lis(\Bun_n^b,L), ~ \pi \mapsto \mathcal{F}_\pi,
$$
where $\mathrm{Rep}_L^\infty(G_b(E))$ is the abelian category of smooth $L$-representations of the locally profinite group $G_b(E)$.

\begin{theorem}[{cf.\ \Cref{sec:proof-fargues-conjecture-irreducible-case}}]
  \label{sec:introduction-2-main-result-on-fargues-sheaf}
  Assume that $\mathbb{L}$ is an irreducible, continuous $W_E$-representation over $\Qlbar$. Then there exists a non-zero Hecke eigensheaf $\mathrm{Aut}_{\mathbb{L}}\in D_\lis(\Bun_{n},\Qlbar)$ with eigenvalue $\mathbb{L}$. It is supported on the semi-stable locus and for each $b\in B(G)$ basic,
    \[
      j^\ast_b \Aut_{\mathbb{L}} \cong \mathcal{F}_{\mathrm{LL}_b(\mathbb{L})},
    \]
    where $\mathrm{LL}_b(\mathbb{L})$ is the (generalized Jacquet-)Langlands correspondent of
    $\mathbb{L}$ (a smooth irreducible $\Qlbar$-representation of $G_b(E)$).
 
\end{theorem}

The existence of an Hecke eigensheaf as in \Cref{sec:introduction-2-main-result-on-fargues-sheaf} is obtained by exploiting the spectral action of Fargues-Scholze. Let $X_{\GL_n,\Qlbar}$ be the Artin stack over $\Qlbar$ of $n$-dimensional $W_E$-representations. Let
\[
  \mathrm{IndPerf}(X_{\GL_n,\Qlbar})\times D_\lis(\Bun_n,\Qlbar)\to D_\lis(\Bun_n,\Qlbar), (\mathcal{A},\mathcal{F})\mapsto \mathcal{A}\ast \mathcal{F}
\]
be the spectral action, and let $\mathbb{L}$ be an irreducible $n$-dimensional continuous representation of $W_E$.
The representation $\mathbb{L}$ defines a smooth point $x\in \mathrm{IndPerf}(X_{\GL_n,\Qlbar})$, and we denote by $k(x)_{\mathrm{reg}}$ the regular representation of the residual gerbe at $x$. It is then rather formal to see that for each $\mathcal{F}\in D_{\lis}(\Bun_n,\Qlbar)$ the object
\[
  k(x)_{\mathrm{reg}}\ast \mathcal{F}
\]
is a Hecke eigensheaf with eigenvalue $\mathbb{L}$. Moreover, it must be supported on the semi-stable locus and given by supercuspidal representations there. However, it has no reason to be non-zero in general.

For $\mathcal{F}$ the Whittaker representation supported at the trivial rank $n$ bundle, we show that this Hecke eigensheaf is indeed non-zero and irreducible. Here, we make use again of an averaging functor, this time
\[
  \mathrm{Av}_{\mathbb{L}^\vee}\colon D_\lis(\Bun_n,\Qlbar)\to D_\lis(\Bun_n,\Qlbar)
\]
with $\mathbb{L}^\vee$ the dual representation of $\mathbb{L}$. Namely, $\mathrm{Av}_{\mathbb{L}^\vee}$ agrees with the spectral action of the sum of two simple skyscraper sheaves supported at $x$ and some cyclotomic twist of $\mathbb{L}$. This makes the action of $k(x)_{\mathrm{reg}}$ more tractable, and allows us to invoke the local Langlands correspondence and its known realization in the $\ell$-adic cohomology of the Lubin-Tate tower to show the desired non-vanishing. From here \Cref{sec:introduction-2-main-result-on-fargues-sheaf} follows rather formally, by combining the above with recent results of Hansen-Kaletha-Weinstein \cite{kaletha_weinstein_on_the_kottwitz_conjecture_for_local_shimura_varieties} and Hansen \cite{hansensupercuspidal}.

\subsection{Acknowledgements}
\label{sec:acknowledgments}
Special thanks go to Peter Scholze who patiently answered all our questions, especially those regarding \cite{fargues2021geometrization} when the preprint was not available. We moreover thank Sebastian Bartling, David Hansen, Eugen Hellmann and Dennis Gaitsgory for interesting discussions related to this paper. We also thank Laurent Fargues for his comments on a very early version of this paper and Marie-France Vign\'eras for pointing out an unnecessary hypothesis.

\section{Notation}
\label{sec:notation}

Throughout the text, we fix the following notations:
\begin{itemize}
\item $p$ is a prime,
\item $E$ is a local field with residue field $\F_q$ of characteristic
  $p$ and uniformizer $\pi$,
 \item $\breve{E}$ the completion of the maximal unramified extension of $E$,
\item $C$ is a completed algebraic closure of $E$,
  \item $W_E$ the Weil group of $E$ (with respect to $C$),
\item $k$ the residue field of $C$ ($k$ is therefore an algebraic closure of $\F_q$),
\item $\Perf_k$ is the category of perfectoid spaces over $k$,
\item $\ell$ is a prime, $\ell\neq p$,
\item $\Lambda$ is a \textit{torsion} $\Z_\ell[\sqrt{q}]$-algebra, i.e., $\ell^i\Lambda=0$ for some $i\geq 1$,
\item $L$ some $\Z_\ell[\sqrt{q}]$-algebra, e.g., $L=\Qlbar$,
\item $G:=\GL_n$ over $E$,
\item $\hat{G}:=\GL_n$ over $\Lambda$ (or occasionally $L$),
\item All our tensor products are derived, that is we constantly write $-\otimes-$ instead of $-\otimes^L-$.  
\end{itemize}

We note that we assume that $\Lambda$ or $L$ contains a fixed square root of $q$. This is necessary when we use Hecke operators via the geometric Satake equivalence.

\section{Geometrization of local Langlands: recollection}
\label{sec:farg-conj-gl_n}

\subsection{The stack of vector bundles on the Fargues-Fontaine curve}
\label{sec:stack-of-vector-bundles-on-the-ff-curve}
Let
\[
  \Bun_n
\]
be the small $v$-stack sending $S\in \Perf_k$ to the groupoid of vector
bundles of rank $n$ on the Fargues--Fontaine curve $X_S=X_{S,E}$
relative to $S$ (and the local field $E$). The geometry of the stack $\Bun_n$ has been extensively studied in \cite{fargues2021geometrization}. Let us recall in particular the following important facts:
\begin{enumerate}
\item There is a map
  $G(\breve{E})\to \Bun_{n}(k),\ b\mapsto \mathcal{E}_b$
  inducing a bijection
  \[
    B(G):=G(\breve{E})/{\sigma-\textrm{conjugacy}}\cong
    |\Bun_n|.
  \]
\item For $b\in B(G)$, write
  \[
    \Bun_n^b
  \]
  for the (locally closed) substack of vector bundles, which are
  $v$-locally isomorphic to $\mathcal{E}_b$. The locally closed
  inclusions
  \[
    j_b\colon \Bun_n^b\hookrightarrow \Bun_n
  \]
  induce a stratification of $\Bun_n$.
\item The vector bundle $\mathcal{E}_b$ is semistable if and only if $b$ basic. The
 semistable locus $\Bun_n^{\rm ss}$ is open in $\Bun_n$ and has the description
  \[
    \coprod\limits_{d\in \Z}[\ast/\underline{G_b(\Q_p)}]\cong
    \coprod\limits_{b\in B(G) \textrm{ basic }}
    \Bun_n^b=\Bun_n^{\mathrm{ss}},
  \]
  where $\kappa(b)=\mathrm{deg}(\mathcal{E}_b)=d$, $G_b$ is the
  $\sigma$-stabilizer of $b$ (an inner form of
  $G$) and
  \[
    \underline{G_b(E)}\cong \mathrm{Aut}(\mathcal{E}_b)
  \]
  the sheaf on $\Perf_k$ associated to the topological group
  $G_b(E)$. We denote by
  \[
    j\colon \Bun_n^{\rm ss} \hookrightarrow \Bun_n
  \]
  the open immersion of the semistable locus in $\Bun_n$.
\end{enumerate}

For each perfectoid space $S$, each $\mathcal{E}\in \Bun_m(S),\ m\geq 0$, gives rise to a small
$v$-sheaf\footnote{The notation $\mathrm{BC}$ stands for \textit{Banach-Colmez space}. We warn the reader that there are various closely related objects called Banach-Colmez spaces in the literature in this setting, and that we simply use here this notation for convenience.}
\[
  \mathrm{BC}(\mathcal{E})\colon \Perf_S\to (E-v.s.),\ T\mapsto
  H^0(X_{T,E},\mathcal{E}_T).
\]
Set
\[
  \Div^d:=(\mathrm{BC}(\mathcal{O}(d))\setminus
  \{0\})/\underline{E^\times}
\]
where $\mathcal{O}(d)\in \Bun_1(k)$ is the line bundle associated to
$b=\pi^{-d}\in \GL_1(\breve{E})$. The v-sheaf $\Div^d$ parametrizes ``relative Cartier effective divisors of
  $X_{E}$ of degree $d$''. As v-sheaves,
  \[  
    \Div^d=(\Div^1)^d/{S_d},
  \]
  where $S_d$ denotes the symmetric group on $d$ letters. When $d=1$, the correspondence between classical points on the Fargues-Fontaine curve and untilts provide an explicit description of $\mathrm{Div}^1$, cf.\ \cite[Definition II.1.19]{fargues2021geometrization}. Namely,
  \[
    \Div^1=\mathrm{Spd}(\breve{E})/\varphi_{\mathrm{Spd}(\breve{E})}^\Z.
  \]
  This descriptions shows in particular that the category of finite rank $L$-local systems on $\Div^1$ is equivalent to the category of finite degree continuous $L$-representations of the Weil group $W_E$ (here, continuity is understood with respect to the topology on $L$ defined as the colimit topology of the $\ell$-adic topology on its finitely generated $\Z_\ell$-submodules). For this reason, we will largely not differentiate in the notation between representations of $W_E$ and the associated sheaves on $\Div^1$.

\subsection{The category of lisse \texorpdfstring{$\ell$}-adic sheaves}
\label{sec:category-of-lisse-l-adic-sheaves}
Defining the correct category of sheaves on the geometric objects we just introduced is a delicate problem. For a small Artin $v$-stack $Y$\footnote{The hypothesis that $Y$ is an Artin v-stack is not needed to define $D_\bs(Y,L)$.}, Fargues-Scholze define full
subcategories
\[
  D_{\lis}(Y,L)\subseteq D_{\bs}(Y,L) \subseteq D(Y_v,L)
\]
and various operations on them. The first one will be the more relevant to us, but is constructed through the intermediate category $D_{\bs}(Y,L)$ of \textit{solid sheaves of $L$-modules on $Y$}. Here again, we refer the reader to \cite{fargues2021geometrization}, in particular \cite[Chapter VII]{fargues2021geometrization}, for a detailed discussion. 
Let us only recall the following key properties:
\begin{enumerate}
\item If $\ell^mL=0$ for some $m\geq 1$, then
  \[
    D_\lis(Y,L)=D_\et(Y,L),
  \]
  with the right hand side the category defined in \cite{scholze_etale_cohomology_of_diamonds}.
\item Excision holds on $\Bun_n$, i.e., $D_\lis(\Bun_n,L)$ admits
  a(n infinite) semi-orthogonal decomposition by the subcategories
  $D_{\lis}(\Bun_n^b,L), b\in B(G)$.
\item For $b\in B(G)$ there are equivalences
  \[
    \begin{matrix}
      D(\mathrm{Rep}_{L}^\infty G_b(E)) & \cong & D_\lis([\ast/\underline{G_b(E)}],L)&\cong & D_\lis(\Bun_n^b,L) \\
      \pi& &\mapsto & & \mathcal{F}_\pi
    \end{matrix}
  \]
  with $\mathrm{Rep}_{L}^\infty G_b(E)$ the category of
  \textit{smooth} $L$-representations of $G_b(E)$. Hence, the category $D_\lis(\Bun_n,L)$ provides a natural way to geometrize the category of smooth representations on $L$-modules of $G(E)$ and its inner forms.
\item Let $f: Y^\prime \to Y$ be a morphism of small Artin v-stacks. The usual pullback functor $f_v^\ast$ attached to the morphism $f_v\colon Y^\prime_v\to Y_v$ induce functors
\[
  f^\ast\colon D_\bs(Y,L)\to D_\bs(Y^\prime,L), ~ f^\ast\colon D_\lis(Y,L)\to D_\lis(Y^\prime,L).
\]
\item Let $f: Y^\prime \to Y$ be a morphism of small v-stacks. The functor $f^\ast\colon D_\bs(Y,L)\to D_\bs(Y^\prime,L)$ has a \textit{left} adjoint
$$
f_\natural \colon D_\bs(Y^\prime,L)\to D_\bs(Y,L).
$$
Moreover, if $f$ is $\ell$-cohomologically smooth and representable in locally spatial diamonds, then
  \[
    f_\nat(D_\lis(Y^\prime,L))\subseteq D_\lis(Y,L).
  \]
  In this case we set
  \[
    f_!(-):=f_{\nat}(-\otimes_L (f^{!}L)^\vee)\colon D_\bs(Y^\prime,L)\to D_\bs(Y,L),
  \]
  where $f^!L:=L\otimes_{\Z_\ell} \varprojlim\limits_{m}f^!\Z/\ell^m\in D_\bs(Y^\prime,L)$ denotes the dualizing complex. If $L$ is $\ell^m$-torsion for some $m\geq 1$ (and $f$ representable in locally spatial diamonds), then this recovers the functor $f_!$ constructed in \cite[Section 22]{scholze_etale_cohomology_of_diamonds}.
\item If $f\colon [\ast/\underline{G_b(E)}]\to \ast$, then $f_\nat$
  identifies with group \textit{homology} of smooth
  $L$-representations of $G_b(E)$.
\end{enumerate}

\subsection{Statement of Fargues' conjecture for $\GL_n$ }
\label{sec:statement-of-fargues-conjecture}
  After these preliminaries, we can now state Fargues' conjecture, in the special case relevant to us. 
  
\begin{conjecture}[Fargues, only for $\GL_n$ and in the irreducible case]
  \label{conjecture-fargues}
  For each irreducible continuous $\Qlbar$-representation $\mathbb{L}$ of $W_E$ of degree $n$, there exists an
  object $\mathrm{Aut}_{\mathbb{L}} \in D_\lis(\Bun_n,\Qlbar)$ such that
  \begin{enumerate}
  \item $\Aut_{\mathbb{L}}$ is a Hecke eigensheaf with eigenvalue
    $\mathbb{L}$ (see below for the exact meaning).
  \item $\Aut_{\mathbb{L}}$ is cuspidal; in particular
    $\Aut_{\mathbb{L}} \cong j_!(j^\ast \Aut_{\mathbb{L}})$.
  \item For $b\in B(G)$ basic,
    \[
      j^\ast_b \Aut_{\mathbb{L}} \cong \mathcal{F}_{\mathrm{LL}_b(\mathbb{L})},
    \]
    where $\mathrm{LL}_b(\mathbb{L})$ is the (generalized Jacquet-)Langlands correspondent of
    $\mathbb{L}$ (a smooth irreducible $\Qlbar$-representation of $G_b(E)$).
  \end{enumerate}
\end{conjecture}

\begin{remark}
\label{remark-formula-for-aut-L}
Note that
(2), (3) force
\begin{equation}\tag{$*$}
  \label{eq:1}
  \mathrm{Aut}_{\mathbb{L}} \cong \bigoplus\limits_{b\in B(G) \textrm{ basic}} j_{b,!}(\mathcal{F}_{\mathrm{LL}_b({\mathbb{L}})}), 
\end{equation}
hence $\mathrm{Aut}_{\mathbb{L}}$ is unique up to isomorphism. One could take this formula as a possible definition of $\mathrm{Aut}_{\mathbb{L}}$, but this would make the Hecke
eigensheaf property (1) hard to verify.
\end{remark}

Let us make more precise the meaning of the Hecke eigensheaf property. Fargues' conjecture is formulated for $L=\Qlbar$, but the Hecke property makes sense for any $\Z_l$-algebra $L$ as ring of coefficients, so let us formulate it in this generality. Recall that $\hat{G}=\GL_{n,L}$. By the geometric Satake equivalence, the category
\[
  \mathrm{Rep}(\hat{G})
\]
of finite dimensional (algebraic) representations of $\hat{G}$ acts on
$D_{\lis}(\Bun_n,L)$, cf. \cite[Proposition IX.2.1.]{fargues2021geometrization}. Even better, for any finite set $I$,
there exists a functor
\[
  T^I\colon \mathrm{Rep}(G)^{I}\times
  D_{\lis}(\Bun_n,L)\to D_{\bs}(\Bun_n\times
  (\Div^1)^I,L),\ (V,\mathcal{F})\mapsto T_{V}^I(\mathcal{F}),
\]
depending functorially on the finite set $I$.

For an $n$-dimensional $\ell$-adic representation $\mathbb{L}$ of $W_{E}$ and a(n algebraic)
representation $\hat{G}_L \to \GL(V)$, set $r_{V,\ast} (\mathbb{L})$ as the
composition
\[
  W_{E}\to \GL(\mathbb{L})\cong \hat{G}(L)\to \GL(V).
\]
Now, $\mathcal{F}\in D_{\lis}(\Bun_n,\Qlbar)$ is a Hecke eigensheaf
with eigenvalue $\mathbb{L}$ if for any finite set $I$ and any
collection $(V_i)_{i\in I}$ we are given an isomorphism
\[
  \eta_{(V_i)_{i\in I}}\colon T_{(V_i)_{i\in I}}^I(\mathcal{F})\cong
  \mathcal{F}\boxtimes (\boxtimes_{i\in I} r_{V_i,\ast}(\mathbb{L})),
\]
which is natural in $I$, $(V_i)_{i\in I}$, and compatible with composition/exterior tensor product.  E.g., if $I=\{1\}$ and
$V_{\mathrm{std}}$ the standard representation, then
\[
  \eta_{V_{\mathrm{st}}}\colon T_{V_{\mathrm{st}}}(\mathcal{F})\cong
  \mathcal{F}\boxtimes \mathbb{L}.
\]

The meaning of the cuspidality condition in condition (2) of the conjecture will be clarified below (see \Cref{definition-cuspidal-object}). 

\subsection{Fargues-Scholze's construction of the spectral action}
\label{sec:work-of-fargues-scholze}
In this short subsection, we recall one of the main results of \cite{fargues2021geometrization} (specialized to the case $G=\GL_n$), since we will use it later. 

As in the last section, we denote $\hat{G}=\GL_{n,L}$. For each finite set $I$ and each $V \in \mathrm{Rep}(\hat{G})^{I}$, the essential image of the Hecke operator
$$
T_V^I: D_\lis(\Bun_n,L) \to D_\bs(\Bun_n \times (\mathrm{Div}^1)^I, L)
$$
introduced above lies inside the essential image of $D_\bs(\Bun_n \times [\ast/W_E^I], L)$ (recall the map $\mathrm{Div}^1=[\mathrm{Spd}(C)/W_E] \to [\ast/W_E]$) and its pullback to $D_\bs(\Bun_n,L)$ lies in $D_\lis(\Bun_n,L)$. In other words, $T_V^I$ can be seen as a functor
$$
T_V^I: D_\lis(\Bun_n,L) \to D_\lis(\Bun_n,L)^{W_E^I}
$$
into the category of objects in $D_\lis(\Bun_n,L)$ endowed with a $W_E^I$-action. This follows by induction from the key invariance
$$
D_\lis(\Bun_n,L) \cong D_\lis(\Bun_n \times \mathrm{Spd}(C),L),
$$
which is itself a consequence of the semi-orthogonal decomposition of $D_\lis(\Bun_n,L)$ by the $D_\lis(\Bun_n^b,L)$, $b \in B(G)$ (point (1) of \Cref{sec:category-of-lisse-l-adic-sheaves}) and the analogous invariance statement for the classifying stacks $[\ast/G_b(E)]$ (via point (2) of \Cref{sec:category-of-lisse-l-adic-sheaves}). 
\\

The existence of such a functorial collection of operators $T_V^I$, which is already enough to provide the construction of semi-simple $L$-parameters attached to smooth irreducible representations of the groups $G_b(E)$, see \cite[Chapter VIII.4.]{fargues2021geometrization}, can be reformulated in a different, more geometric, manner.

 Let
\[
  X_{\hat{G}}
\]
be the algebraic stack over $L$ of $n$-dimensional continuous $L$-representations of
$W_E$, defined and studied in \cite[VIII.1]{fargues2021geometrization} or \cite{dat_helm_kurinczuk_moss}, \cite{zhu_coherent_sheaves_on_the_stack_of_langlands_parameters}. The algebraic stack $X_{\hat{G}}$ comes with a map
$$
f\colon  X_{\hat{G}} \to [\mathrm{Spec}(L)/\hat{G}].
$$
Note that $X_{\hat{G}}$ is an infinite disjoint union of finite type Artin stacks. Therefore, it could be seen as a stack locally of finite type, but we rather want to see it as an ind-finite type stack. In particular, coherent sheaves on it are required to have support on finitely many connected components, and thus we see for example the structure sheaf of $X_{\hat{G}}$ not as a vector bundle, but as an ind-vector bundle (cf. \cite[Remark 3.1.7]{zhu_coherent_sheaves_on_the_stack_of_langlands_parameters}). 

The following is \cite[Theorem X.0.1]{fargues2021geometrization}.

\begin{theorem}
\label{main-theorem-fargues-scholze-existence-of-the-spectral-action}
The geometric Hecke action gives rise to an $L$-linear monoidal action of the category
$$
\mathrm{Perf}(X_{\hat{G}})
$$
of perfect complexes on $X_{\hat{G}}$ on the full sucategory $D_\lis(\Bun_n,L)^{\omega}$ of $D_\lis(\Bun_n,L)$ formed by compact objects, such that for any finite set $I$, and any $V=(V_i)_{i\in I} \in \mathrm{Rep}(\hat{G})^I$, the functor $T_V^I$ is given by the action of the perfect complex
$$
\bigotimes_{i\in I} f^\ast V_i,
$$
where for each $i\in I$, we denoted again by $V_i$ the vector bundle on $[\mathrm{Spec}(L)/\hat{G}]$ attached to $V_i$. 
\end{theorem} 

Passing to ind-completions (recall that $D_\lis(\Bun_n,L)$ is compactly generated, \cite[Proposition VII.7.4]{fargues2021geometrization}), one gets an $L$-linear monoidal action of $\mathrm{IndPerf}(X_{\hat{G}})$ on $D_\lis(\Bun_n,L)$. 

We will use the following notation. If $A \in \mathrm{IndPerf}(X_{\hat{G}})$, and $\mathcal{F} \in D_\lis(\Bun_n,L)$, we will denote by
$$
A \ast \mathcal{F}
$$
the object in $D_\lis(\Bun_n,L)$ obtained by letting $A$ act on $\mathcal{F}$.

We note that if $V\in \mathrm{Rep}(\hat{G})$, then $f^\ast V\in \Perf(X_{\hat{G}})$ has a canonical $W_E$-action, and the isomorphism
  \[
    f^\ast V\ast \mathcal{F}\cong T_V(\mathcal{F})
  \]
is $W_E$-equivariant for each $\mathcal{F}\in D_\lis(\Bun_G,L)$.

\section{Constant term and Eisenstein functors}
\label{sec:const-term-eisenst}

We use the notations introduced in \Cref{sec:notation}.
Let $P \subset G$ be a standard parabolic subgroup, with Levi
quotient $M$. The embedding of $P$ in $G$ and the projection
$P \to M$ induce morphisms \footnote{We note that there is a conflict of notation, $p$ denotes the residue characteristic of $E$ and the morphism $\Bun_P\to \Bun_n$. We think that this will not create confusion.
}
\[
  \Bun_n \overset{p} \longleftarrow \Bun_P \overset{q}
  \longrightarrow \Bun_M.
\]
We record the following properties.

\begin{lemma}
  \label{sec:const-term-eisenst-1-properties-of-p-q}
  \begin{enumerate}
  \item The morphism $p\colon \Bun_P\to \Bun_n$ is representable in locally spatial diamonds, compactifiable and locally on $\Bun_P$ of $\mathrm{dimtrg}(p)<\infty$.
    \item The morphism $q\colon \Bun_P\to \Bun_M$ is a cohomologically smooth morphism of smooth Artin $v$-stacks in the sense of \cite[Definition IV.1.11.]{fargues2021geometrization}.  
  \end{enumerate}  
\end{lemma}

In particular, a natural functor $Rq_!\colon D_\et(\Bun_P,\Lambda)\to D_\et(\Bun_M,\Lambda)$ is defined in \cite[Remark IV.1.15.]{fargues2021geometrization}, although $q$ is not representable (in locally spatial diamonds).

\begin{proof}
  Let $S\in \Perf_{k}$ be a perfectoid space and let $\mathcal{E}\in \Bun_n(S)$ be a vector bundle of rank $n$ on $X_S$.
  The fiber product
  \[
    \Bun_P\times_{\Bun_n}S
  \]
  parametrizes reductions of the associated $G$-torsor $\tilde{\mathcal{E}}$ of $\mathcal{E}$ to $P$ and can therefore be written as the space
  \[
    \mathcal{M}_Z
  \]
  associated with
  \[
    Z:=\tilde{\mathcal{E}}/P
  \]
  living over $X_S$, cf.\ \cite[Section IV.4.]{fargues2021geometrization}. By \cite[Theorem IV.4.2.]{fargues2021geometrization} the morphism
\[
  p\colon \Bun_P\to \Bun_M
\]
is therefore representable in locally spatial diamonds and compactifiable. If $U\subseteq \Bun_P, V\subseteq \Bun_n$ are quasi-compact with $p(U)\subseteq V$, then it is easy to see that the restriction
\[
  f:=p_{|U}\colon U\to V 
\]
has finite geometric transcendence degree, because its fibers embed into fibrations of Banach-Colmez spaces.

We turn to a proof of the claim about $q$. For this let $S\in \Perf_{k}$ be a perfectoid space and $\mathcal{F}\in \Bun_M(S)$. 
The fiber product
\[
  \Bun_P\times_{\Bun_M} S
\]
is an iterated fibration of $v$-sheaves
\[
  \mathcal{E}xt^1_{X_T}(\mathcal{E}_1,\mathcal{E}_2)
\]
for small $v$-stacks $T$ cohomologically smooth over $S$. More precisely, if $M\cong \GL_{r_1}\times \ldots, \GL_{r_m}$ with $\sum\limits_{i=1}^{m} r_i=n$, then $\mathcal{F}\in \Bun_M(S)$ corresponds to vector bundles
\[
  \mathcal{E}_1,\ldots, \mathcal{E}_m
\]
of ranks $r_1,\ldots, r_m$.
Set $T_0:=S$ and
\[
  T_i:=\mathcal{E}xt^1_{X_{T_{i-1}}}(\mathcal{F}_{i,T_{i-1}},\mathcal{E}_{i+1,T_{i-1}})
\]
for $i=1,\ldots, m-1$, where $\mathcal{F}_{1,T_0}=\mathcal{E}_1$ and $\mathcal{F}_{i,T_{i-1}}$ denotes the universal extension of $\mathcal{E}_{i,T_{i-2}}$ by $\mathcal{F}_{i-1,T_{i-2}}$ on $T_{i-1}$ for $2\leq i\leq m-1$, the $\mathcal{E}xt^1$ the fibered category {over the $v$-site of $X_{T_{i-1}}$} with fiber {over some perfectoid space $S\in X_{T_{i-1},v}$} the \textit{groupoid} of extensions of vector bundles {on the Fargues-Fontaine curve $X_S$}.
Then
\[
  T_m\cong \Bun_P\times_{\Bun_M}S,
\]
and thus the claim follows from \Cref{sec:const-term-eisenst-2-ext-spaces-of-vector-bundles} below. This finishes the proof.
\end{proof}

\begin{lemma}
  \label{sec:const-term-eisenst-2-ext-spaces-of-vector-bundles}
  Let $S\in \Perf_{k}$ and let $\mathcal{E}_1,\mathcal{E}_2$ be vector bundles on $X_S$. Then the fibered category $Y:=\mathcal{E}xt^1_{X_S}(\mathcal{E}_1,\mathcal{E}_2)$ sending $T\in \Perf_S$ to the groupoid of extensions
  \[
    0\to \mathcal{E}_2\to \mathcal{E}\to \mathcal{E}_1\to 0
  \]
  is represented by an Artin $v$-stack, which is cohomologically smooth over $S$.
\end{lemma}
In contrast, the coarse moduli space of $Y$, i.e., the small $v$-sheaf
\[
  T\mapsto H^1(X_T,\mathcal{E}_1^\vee\otimes_{\mathcal{O}_{X_T}} \mathcal{E}_2)
\]
need not be smooth over $S$. E.g., if $s\in S$ is a closed point and $\mathcal{O}\oplus \mathcal{O}$ on $S\setminus\{s\}$ degenerates to $\mathcal{O}(1)\oplus \mathcal{O}(-1)$ in $s$, then the morphism of the coarse moduli space to $S$ is not open, and thus not cohomologically smooth. 
\begin{proof}
  As $S\mapsto \Bun(X_S)$ is a $v$-stack on $\Perf_{k}$ the fibered category $Y$ is a $v$-stack, too.
  We may assume that $S$ is affinoid, and in particular quasi-compact.
  Consider a surjection
  \[
    \mathcal{O}_{X_S}(-d)^N\to \mathcal{E}_2^\vee
  \]
  with $d\gg 0$ (and kernel automatically a vector bundle), and the dual resolution
  \[
    0\to \mathcal{E}_2\to \mathcal{O}_{X_S}(d)^{N}\to \mathcal{F}\to 0 
  \]
  with $\mathcal{F}$ some vector bundle.
  By enlarging $d$ we may (using again quasi-compacity of $S$) assume that
  \[
    \mathcal{E}_1^\vee \otimes_{\mathcal{O}_{X_S}} \mathcal{O}_{X_S}(d)^N,\quad \mathcal{E}_1^\vee\otimes_{\mathcal{O}_{X_S}} \mathcal{F} 
  \]
  have positive Harder-Narasimhan slopes over all of $S$. Indeed, this is clear for $\mathcal{E}_1^\vee\otimes_{\mathcal{O}_{X_S}}\mathcal{O}_{X_S}(d)^N$ and then follows for $\mathcal{E}_1^\vee\otimes_{\mathcal{O}_{X_S}} \mathcal{F}$ because fiberwise on $S$ the $H^1$ of this sheaf vanishes as it is surjected on by the vanishing $H^1$ of $\mathcal{E}_1^\vee\otimes_{\mathcal{O}_{X_S}}\mathcal{O}_{X_S}(d)^N$.
 
  In particular, the small $v$-sheaves
  \[
    X_1:=\mathcal{H}om_{X_S}(\mathcal{E}_1,\mathcal{O}_{X_S}(d)^N),\quad X_2:=\mathcal{H}om_{X_S}(\mathcal{E}_1,\mathcal{F})
  \]
  are positive Banach-Colmez spaces and therefore cohomologically smooth by \cite[Proposition II.3.5.(iii)]{fargues2021geometrization}.
  Now, $Y$ is the stacky quotient of $X_2$ by the action of $X_1$ on it coming from the natural morphism
  \[
   X_1\to X_2. 
  \]
  This finishes the proof.
\end{proof}

We define the ($!$-version of the) constant term functor as
\[
  \CT_{P,!} \colon D_\et(\Bun_n,\Lambda)\to D_\et(\Bun_M,\Lambda)
\]
by the formula
\[
  \CT_{P,!}(\mathcal{G}) := Rq_! p^* \mathcal{G}
\]
for
\[
  \mathcal{G}\in D_\et(\Bun_n,\Lambda).
\]\
When $P$ is the parabolic attached to a partition $n=n_1+n_2$ of $n$,
we sometimes write $\mathrm{CT}_{(n_1,n_2),!}$ instead of
$\mathrm{CT}_{P,!}$.  The functor $\CT_{P,!}$ is left adjoint to
the functor
\[
  \Eis_{P,\ast}:=Rp_\ast\circ Rq^!\colon D_\et(\Bun_M,\Lambda)\to D_\et(\Bun_n,\Lambda)
\]
of ``geometric Eisenstein series''.

The main use of the constant term functors is to define the important category of cuspidal objects.

\begin{definition}
  \label{definition-cuspidal-object}
  We say that an object $\mathcal{G}\in D_\et(\Bun_n,\Lambda)$ on $\Bun_n$
  is \textit{cuspidal} if for any choice of proper standard parabolic
  $P$ with Levi quotient $M$,
  \[
    \CT_{P,!}(\mathcal{G}) =0.
  \]
  We denote by
  \[
  D_{\et,\mathrm{cusp}}(\Bun_n,\Lambda)\subseteq D_\et(\Bun_n,\Lambda)
  \]
  the full triangulated subcategory of cuspidal objects.
\end{definition}

It is equivalent to demand the same condition for all proper, \textit{maximal} standard parabolics $P\subseteq G$, i.e., those associated with a partition $n=n_1+n_2$.

  \subsection{Constant terms and twisting}
\label{sec:const-terms-twist}

Let $P\subseteq G$ be a parabolic with Levi quotient $M$.
As the kernel of
\[
  \mathrm{det}\colon G\to \Gm=\GL_1
\]
contains the unipotent radical of $P$ we get a canonical factorization
\[
  \mathrm{det}_M\colon M\to \Gm,
\]
whose restriction to $P$ agrees with the restriction $\mathrm{det}_P\colon P\to \Gm$ of $\mathrm{det}$. From this we can deduce the existence of a commutative diagram
\begin{equation}
  \label{eq:2-diagram-with-determinants}
  \xymatrix{
    &\Bun_P\ar[ld]^{p}\ar[rd]^q\ar[dd]^{\mathrm{det}_P} & \\
    \Bun_n\ar[rd]^{\mathrm{det}} & & \Bun_M\ar[ld]^{\mathrm{det}_M} \\
    & \Bun_1 &
  }
\end{equation}

We obtain the following compatibility of constant term functors with ``twisting by characters''.

\begin{lemma}
  \label{sec:const-terms-twist-1-constant-terms-and-twisting}
  Let $\mathcal{G}\in D_\et(\Bun_1,\Lambda)$. Then for any $\mathcal{F}\in D_\et(\Bun_n,\Lambda)$ there is a natural isomorphism
  \[
    \CT_{P,!}(\mathcal{F}\otimes_\Lambda \mathrm{det}^\ast\mathcal{G})\cong \CT_{P,!}(\mathcal{F})\otimes_\Lambda \mathrm{det}^\ast_M \mathcal{G}.
  \]
  In particular, $D_{\et,\mathrm{cusp}}(\Bun_n,\Lambda)$ is stable under the functor $-\otimes_{\Lambda} \mathrm{det}^\ast \mathcal{G}$.
\end{lemma}
\begin{proof}
  Using (\Cref{eq:2-diagram-with-determinants}) we can calculate
  \[
    \begin{matrix}
      & \CT_{P,!}(\mathcal{F}\otimes_\Lambda \mathrm{det}^\ast\mathcal{G}) \\
      \cong & Rq_!(p^\ast(\mathcal{F}\otimes_\Lambda \mathrm{det}^\ast\mathcal{G})) \\
      \cong & Rq_!(p^\ast(\mathcal{F})\otimes_\Lambda \mathrm{det}^\ast_P\mathcal{G})) \\
      \cong & Rq_!(p^\ast(\mathcal{F})\otimes_\Lambda q^\ast\mathrm{det}^\ast_M \mathcal{G})) \\
      \cong & \CT_{P,!}(\mathcal{F})\otimes_{\Lambda}\mathrm{det}_M^\ast\mathcal{G},
    \end{matrix}
  \]
  where the last isomorphism comes from the projection formula.
\end{proof}

\begin{lemma}
  \label{sec:const-terms-twist-1-no-rank-one-stuff-cuspidal}
  Let $\mathcal{G}\in D_\et(\Bun_1,\Lambda)$. If $n\geq 2$ and $\mathrm{det}^\ast \mathcal{G}$ is cuspidal, then
  $\mathcal{G}=0$.
\end{lemma}
\begin{proof}
  From the proof of \Cref{sec:const-term-eisenst-1-properties-of-p-q} we can conclude that
  \[
    Rq_!q^!\cong \mathrm{Id}_{\Bun_M},
  \]
  as $q$ is a relative fibration with contractible fibers, and $q^!$ commutes with base change to fibers.
  This implies that for each parabolic $P\subseteq G$ the object
  \[
    \mathrm{CT}_{P,!}(\Lambda)
  \]
  is $\Lambda$-local system of rank $1$ on $\Bun_M$.
  If $n\geq 2$, we can choose some \textit{proper} parabolic $P\subseteq G$ and use \Cref{sec:const-terms-twist-1-constant-terms-and-twisting} to see
  \[
    0=\mathrm{CT}_{P,!}(\mathrm{det}^\ast \mathcal{G})=\CT_{P,!}(\Lambda)\otimes_{\Lambda} \mathrm{det}^\ast_M \mathcal{G}.
  \]
  Thus, we obtain that $\mathrm{det}^\ast_M \mathcal{G}=0$, which implies $\mathcal{G}=0$ as $\Bun_M\to \Bun_1$ is surjective.
\end{proof}

\subsection{Geometrically cuspidal representations}
\label{sec:cuspidal-subcategory}

The next results give some properties of cuspidal objects in
$D_\et(\Bun_n,\Lambda)$.

\begin{lemma}
  \label{sec:v_m-1-implies-cuspidal-implies-support-on-non-semistable-locus}
  Let $\mathcal{F}\in D_{\et,\rm cusp}(\Bun_n,\Lambda)$. Then $\mathcal{F}$
  is supported on the semistable locus
  $\Bun_n^{\rm ss}\subseteq \Bun_n$.
\end{lemma}
\begin{proof}
  Let $b\in B(G)$ be a non-semistable point and let $M$ be
  the standard Levi associated to $G_b$. In particular, we can view $b\in B(M)$. Let $P$ be the standard
  parabolic containing $M$. Then the fiber of $\Bun_P\to \Bun_{M}$
  over the $M$-torsor given by the associated graded of the HN-filtration
  on $\mathcal{E}_b$ maps isomorphically to $\Bun_{n}$ with image $b$
  as the Harder-Narasimhan filtration is canonical. Moreover,
    pull back along the morphism
    $\mathrm{Bun}_P^b:=\Bun_P\times_{\Bun_M} \Bun_M^b$ defines an
    equivalence
    $$
    D_\et(\Bun_M^b,\Lambda)\cong D_\et(\Bun_P^b,\Lambda)
    $$
    because the fiber is a classifying stack for a positive Banach-Colmez space, cf.\ \cite[Proposition V.2.1.]{fargues2021geometrization}.
    This implies that the image of
    $j_{b,\ast }(D_\et([\ast/G_b(E)],\Lambda)$ is contained in the image
    of $\mathrm{Eis}_{P,\ast}(D_\et(\Bun_M,\Lambda))$.  Now let
    $\mathcal{F}\in D_\et(\Bun_n,\Lambda)$ be cuspidal. Then by definition
    $$
    \mathrm{CT}_{P,!}(\mathcal{F})=0.
    $$
   Let $\mathcal{G}\in D_\et([\ast/G_b(E)],\Lambda)$. As $j_{b,\ast}(D_\et([\ast/G_b(E)],\Qlbar))$ is in the image of $\mathrm{Eis}_{P,\ast}$, 
   $$
   j_{b,\ast}(\mathcal{G}) = \mathrm{Eis}_{P,\ast}(\mathcal{G}^\prime),
   $$
   for some $\mathcal{G}^\prime \in D_\et([\ast/G_b(E)],\Lambda)$. Hence, by adjunction,
    $$
    \Hom(j_b^\ast
    \mathcal{F},\mathcal{G})=\Hom(\mathrm{CT}_{P,!}(\mathcal{F}),
    \mathcal{G}^\prime)=0.
    $$
    Since this holds for all $\mathcal{G}\in D_\et([\ast/G_b(E)],\Lambda)$, this implies
    $j_{b}^\ast \mathcal{F}=0$ as desired.
\end{proof}

Let $b\in B(G)$ be a basic class. Recall that a smooth representation $\pi$ of $G_b(E)$
is said to be \textit{cuspidal}, if it has vanishing Jacquet modules for
all choices of proper parabolic subgroups. Recall also from \Cref{sec:category-of-lisse-l-adic-sheaves} that we denote by
$$
\mathcal{F}_\pi\in D_\et([\ast/G_b(E)],\Lambda)
$$
the sheaf associated to a representation $\pi\in \mathrm{Rep}^\infty_\Lambda G_b(E)$ .

We introduce the following convenient terminology.

\begin{definition}
  \label{sec:const-term-funct-definition-geometrically-cuspidal}
  We call $\pi\in \mathrm{Rep}^\infty_{\Lambda} G_b(E)$
  geometrically cuspidal, if
  $j_{b,!}(\mathcal{F}_\pi)\in D_\et(\Bun_n,\Lambda)$ is a cuspidal object.
\end{definition}

\begin{proposition}
  \label{sec:const-term-funct-prop-geom-cuspi-implies-supercuspidal}
  Let $\pi\in \mathrm{Rep}^\infty_{\Lambda} G_b(E)$ be geometrically
  cuspidal. Then $\pi$ is cuspidal.
  If moreover $n\geq 2$, then $\pi$ is not the inflation of a representation of $E^\times$ along $\mathrm{Nrd}\colon G_b(E)\to E^\times$. 
\end{proposition}
\begin{proof}
  Let $P_b\subseteq G_b(E)$ be parabolic subgroup with Levi
  quotient $M_b$.
    Recall that $G$ and $G_b$ are pure inner forms over $X$. Thus, as
  $G$ is quasi-split, we can find a parabolic subgroup $P$ with
  Levi quotient $M$ such that
    $$
    \Bun_{P} \cong \Bun_{P_b} ~~ , ~~ \Bun_M \cong \Bun_{M_b}.
    $$
    Let $c\in \Bun_{M}$ be the point determined by the trivial
    $M_b$-torsor under the equivalence $\Bun_{M_b}\cong \Bun_M$. Then
    $$
    \Bun_{P}^c:=\Bun_{P}\times_{\Bun_{M}} \Bun_{M}^c\to \Bun_n\cong
    [\ast/P_b]
    $$
    has image in $\Bun_n^b$. The resulting diagram identifies with
    $$
    [\ast/G_b(\Q_p)]\xleftarrow{p} [\ast/P_b] \xrightarrow{q}
    [\ast/M_b]
    $$
    and thus $\mathrm{CT}_{P,!}$ restricted to $c$ coincides with the
    Jacquet functor for $P_b$.
        The last claim follows from \Cref{sec:const-terms-twist-1-no-rank-one-stuff-cuspidal}.
    This finishes the proof.
  \end{proof}
  
  \begin{remark}
  Assume that $\Lambda$ is an algebraically closed field of characteristic $\ell$. If $\pi$ is a smooth representation $G_b(E)$ over $\Lambda$ which is irreducible, saying that $\pi$ is cuspidal is the same, by Frobenius reciprocity, as saying that $\pi$ is not a subobject of a parabolically induced representation. If the smooth irreducible representation $\pi$ satisfies that it is not a \textit{subquotient} of a parabolically induced representation, then $\pi$ is said to be \textit{supercuspidal}. Cuspidality and supercuspidality are equivalent notions in characteristic $0$, but not modulo $\ell$, where the latter is strictly stronger than the former, cf. e.g. \cite[Corollaire 5]{vigneras1989representations}. 
  \end{remark}

  \begin{example}
    \label{sec:geom-cusp-repr-geom-cusp-representations-for-n-2}
    Assume that $n=2$. Pick $b\in B(G)$ basic with $d:=\kappa(b)\in 2\Z$. If $\pi\in \mathrm{Rep}^\infty_\Lambda G_b(E)$ is irreducible supercuspidal, then $\pi$ is geometrically cuspidal. Indeed, for a fixed $m\geq d$, consider the space
    \[
      Z_{{m}}:=\mathcal{H}om(\mathcal{E}_b,\mathcal{O}({m}))
    \]
    of homomorphisms from $\mathcal{E}_b$, and its open subspace $Z_{{m}}^\circ\subseteq Z_{{m}}$ parametrizing surjections. Over $Z_{{m}}^\circ$ there is the natural $E^\times$-torsor $\widetilde{Z_{{m}}^\circ}$ parametrizing isomorphisms of the kernel {with} $\mathcal{O}(d-{m})$. We have to prove that
    \[
      R\Gamma_{\natural}(G_b(E),R\Gamma_c(\widetilde{Z_{{m}}^\circ},\Lambda)\otimes_{{\Lambda}} \pi)=0
    \]
    But $\widetilde{Z_{{m}}^\circ}$ is isomorphic to $Z^\circ_{m}\times E^\times$ because giving an isomorphism of the kernel with $\mathcal{O}(d-m)$ is an $E^\times$-torsor equivalent to the (split) $E^\times$-torsor of isomorphisms $\mathrm{det}(\mathcal{E}_b^2)\cong \mathcal{O}(d)$. In particular, $G_b(E)$ acts on the factor $E^\times$ via the determinant. One can therefore reduce to proving that 
    \[
        R\Gamma_{\natural}(G_b(E),R\Gamma_c(Z_{{m}}^\circ,\Lambda)\otimes_{\Lambda} \chi \otimes_{{\Lambda}} \pi)=0,
      \]
      for any character $\chi$. For ${m}>d$ this follows by the known cohomology of the punctured BC-space $Z_{{m}}^{\circ}$ and the fact that $R\mathrm{Hom}_{G_b(E)}(\pi,\chi)=0$ for any character $\chi$ by supercuspidality of $\pi$, cf. case (1) in the proof of \cite[Proposition 6.1]{drevon2021block}\footnote{Proposition 6.1 in \cite{drevon2021block} is stated with $\Lambda=\Flbar$, but the discussion of Case (1) in the proof applies without this hypothesis.}. For $m=d$, this follows as supercuspidal representations are cuspidal, i.e., they have vanishing Jacquet functors. Conversely, this argument shows that if $\pi$ is irreducible and geometrically cuspidal, it must be cuspidal (case $m=d$). Moreover, if $\pi$ was not supercuspidal, which can happen only when $q=-1$ modulo $\ell$, then $\pi$ would be a twist by a character $\chi$ of the representation $\pi_1$ from \cite[Corollaire 5]{vigneras1989representations}. In particular, we would have
      $$
      \mathrm{Ext}_{G_b(E)}^1(\pi_1, (-1)^{\mathrm{val}} \chi)\neq 0,
      $$
   contradicting geometric cuspidality of $\pi$ (case $m<d$).
      
      Hence, when $n=2$, for any $b\in B(G)$ basic with $\kappa(b)$ even, a smooth irreducible representation $\pi$ of $G_b(E)$ is geometrically cuspidal if and only if it is supercuspidal. 
       \end{example}

\section{Averaging functors}
\label{sec:averaging-functors}

We now introduce and analyze the averaging functors which are alluded to in the title.
We use the same notation as in \Cref{sec:notation}.

\subsection{The definition of the averaging functors}
\label{sec:defin-aver-funct}

Let $d,n\geq 0$ and let
\[
  \Mod^d_n
\]
be the small $v$-stack on $\Perf_k$ sending $S\in \Perf_k$ to the
groupoid of the data
\[
  (\mathcal{E},\mathcal{E}^\prime,x,\gamma)
\]
with $\mathcal{E},\mathcal{E}^\prime\in \Bun_n(S)$ satisfying
$\kappa(\mathcal{E}^\prime)-\kappa(\mathcal{E})=d$, $x\in \Div^1(S)$,
and
\[
  \gamma\colon \mathcal{E}\hookrightarrow \mathcal{E}^\prime
\]
a morphism of vector bundles, which is fiberwise injective with
cokernel supported at $x$. We will actually only need the case that $d=1$.
Consider the diagram
\begin{equation}
  \label{eq:2-correspondence-defining-averaging-functors}
  \xymatrix{
    &\mathrm{Mod}^1_{n}\ar[ld]^{\overleftarrow{h}}\ar[rd]_{\overrightarrow{h}}\ar[r]^-{\alpha} & \mathrm{Div}^1 \\
    \Bun_n & & \Bun_n
  }
\end{equation}
with
  \[
    \begin{matrix}
      \overleftarrow{h}(\mathcal{E},\mathcal{E}^\prime,x,\gamma)=\mathcal{E}, \\
      \overrightarrow{h}(\mathcal{E},\mathcal{E}^\prime,x,\gamma)=\mathcal{E}^\prime, \\
      \end{matrix}
    \]
  and let
\[
  \mathbb{L}\in D_{\et}(\Div^1,\Lambda)
\]
be any object, e.g., the sheaf associated to a
$\Lambda$-representation of the Weil group $W_E$ of $E$.

Note that $\overleftarrow{h}$ and $\overrightarrow{h}$ are representable in locally spatial diamonds, proper, of finite $\mathrm{dim.trg}$ and cohomologically smooth.

\begin{definition}
  \label{sec:averaging-functors-1-definition-averaging-functor}
  We define the averaging functor associated with $\mathbb{L}\in D_{\et}(\Div^1,\Lambda)$ to be
  \[
    \Av_{\mathbb{L},n}\colon D_\et(\Bun_n,\Lambda)\to
    D_\et(\Bun_n,\Lambda), ~ \mathcal{F}\mapsto
    \overrightarrow{h}_!(\overleftarrow{h}^\ast(\mathcal{F}){\otimes}_\Lambda
    \alpha^{\ast}\mathbb{L}(\frac{n-1}{2})[n-1])
  \]
\end{definition}
If $n$ is clear from the context, then we write
\[
  \mathrm{Av}_{\mathbb{L}}:=\mathrm{Av}_{\mathbb{L},n}.
\]

Thus, the averaging functors provide examples of, potentially interesting, endofunctors of $D_{\et}(\Bun_n,\Lambda)$.

\subsection{Averaging functors and constant terms}
\label{sec:aver-funct-const}

One of the most important properties of the averaging functors is their
commutation with constant terms.

\begin{proposition}
  \label{sec:averaging-functors-2-averaging-functors-and-constant-terms}
  Let $P\subseteq G$ be a maximal standard parabolic, and let
  $M=\GL_{n_1}\times \GL_{n_2}$ be its Levi quotient. Then there is a
  canonical exact triangle
  \[
    (\Av_{\mathbb{L}(\frac{n-n_1}{2})[n-n_1],n_1}\times \mathrm{Id}_{\Bun_{n_2}})\circ
    \CT_{P,!} \to \CT_{P,!}\circ \Av_{\mathbb{L},n}\to
    (\mathrm{Id}_{\Bun_{n_1}}\times \Av_{\mathbb{L}(\frac{n-n_2}{2})[n-n_2],n_2})\circ \CT_{P,!}
  \]
  of functors
  $D_\et(\Bun_n,\Lambda)\to
  D_\et(\Bun_M,\Lambda)=D_\et(\Bun_{n_1}\times \Bun_{n_2},\Lambda)$.
\end{proposition}

Here, we used the notation $\Av_{\mathbb{L},n_1}\times \mathrm{Id}_{\Bun_{n_2}}, \mathrm{Id}_{\Bun_{n_1}}\times \Av_{\mathbb{L},n_2}$ for the functors arising from the correspondence (with kernel the pullback of $\mathbb{L}$) obtained by taking the product of $\eqref{eq:2-correspondence-defining-averaging-functors}$ (for $n=n_1$ resp.\ $n=n_2$) with $\Bun_{n_2}$ resp.\ $\Bun_{n_1}$. For an analog in the classical case we
refer to \cite[Lemma 9.8]{frenkel_gaitsgory_vilonen}.
\begin{proof}
  Let
  \[
    Y:=\Mod^1_n\times_{\overrightarrow{h},\Bun_n}\Bun_P
  \]
  be the stack parametrizing injections
  $\gamma\colon \mathcal{E}\to \mathcal{E}^\prime$ of
  $\mathcal{E},\mathcal{E}^\prime\in \Bun_n$ with
  $\kappa(\mathcal{E}^\prime)=\kappa(\mathcal{E})+1$ and filtrations
  \[
    0\to \mathcal{E}^\prime_1\to \mathcal{E}^\prime\to
    \mathcal{E}^\prime_2\to 0
  \]
  with $\mathcal{E}^\prime_i\in \Bun_{n_i}, i=1,2$.  Let
  \[
    j\colon U\to Y
  \]
  be the open substack defined by the condition that the composition
  $\mathcal{E}\to \mathcal{E}^\prime\to \mathcal{E}^\prime_2$ is
  surjective.  Let
  \[
    i\colon Z\to Y
  \]
  be the closed complement of $U\subseteq Y$.
  Let $\mathcal{K}\in D_{\et}(\Bun_n\times \Bun_M,\Lambda)$ be the kernel representing the composition
  \[
    \CT_{P,!}\circ \Av_{\mathbb{L},n},
  \]
  i.e., $\mathcal{K}$ is the $!$-pushforward for the morphism
  \[
    Y\to \Bun_n\times \Bun_M
  \]
  of the pullback $\mathbb{L}_Y\in D_\et(Y,\Lambda)$ of $\mathbb{L}$ along
  \[
    Y\to \mathrm{Mod}^1_n\to \mathrm{Div}^1.
  \]
  Our aim is to identify the functors 
  \[
    (\Av_{\mathbb{L}(\frac{n-n_1}{2})[n-n_1],n_1}\times \mathrm{Id}_{\Bun_{n_2}})\circ
    \CT_{P,!}, (\Av_{\mathbb{L}(\frac{n-n_2}{2})[n-n_2],n_1}\times \mathrm{Id}_{\Bun_{n_2}})\circ
    \CT_{P,!}\colon D_\et(\Bun_n,\Lambda)\to D_\et(\Bun_M,\Lambda)
  \]
  as the functors with kernels the two outer terms in the triangle
  \[
    j_!j^\ast \mathcal{K}\to \mathcal{K}\to i_\ast i^\ast \mathcal{K}.
  \]
  This will imply the proposition. We start with identifying the functor with kernel $j_!j^\ast\mathcal{K}$. It is calculated via the correspondence
\[
  \xymatrix{
    &U\ar[ld]\ar[rd]\ar[r] & \mathrm{Div}^1 \\
    \Bun_n & & \Bun_M
    }
  \]
  with kernel the pullback $\mathbb{L}_U\in D_\et(U,\Lambda)$ of $\mathbb{L}$. Over $U$ the
  intersection
  \[
    \mathcal{E}_1:=\mathcal{E}\cap \mathcal{E}_1^\prime,
  \]
  i.e., the kernel of the surjection
  $\mathcal{E}\twoheadrightarrow \mathcal{E}^\prime_2$, is again a
  vector bundle. Note that the injection
  $\mathcal{E}_1\hookrightarrow \mathcal{E}_1^\prime$ lies in
  $\mathrm{Mod}_{n_1}^1$.  The morphism
  \[
    U\to \Mod^1_{n_1}\times_{\overleftarrow{h},\Bun_{n_1}}\Bun_P
  \]
  sending the above data to
  $\mathcal{E}_1\hookrightarrow \mathcal{E}_1^\prime$ and the
  filtration
  \[
    0\to \mathcal{E}_1\to \mathcal{E}\to \mathcal{E}_2^\prime\to 0
  \]
  is an equivalence as necessarily $\mathcal{E}^\prime$ is the
  pushout of the diagram of
  \[
    \mathcal{E}\leftarrow \mathcal{E}_1\rightarrow
    \mathcal{E}_1^\prime.
  \]
  In particular, we obtain the diagram is a diagram
  \[
    \xymatrix{
      & & U \ar[ld]\ar[rd]& & \\
      & \Bun_P\ar[ld]\ar[rd]& & \Mod^1_{n_1}\times \Bun_{n_2}\ar[ld]\ar[rd] & \\
      \Bun_n & & \Bun_{n_1}\times \Bun_{n_2} & & \Bun_{n_1}\times \Bun_{n_2}
    }
  \]
  with cartesian square.
  Now, $\mathbb{L}_U\in D_\et(U,\Lambda)$ is the pullback of $\mathbb{L}\in D_\et(\Div^1,\Lambda)$ along the morphism
  \[
    U\to \Mod^1_{n_1}\times \Bun_{n_2}\to \Div^1.
  \]
  This implies that the functor induced by $j_!j^\ast \mathcal{K}$ is $(\Av_{\mathbb{L}(\frac{n-n_1}{2})[n-n_1],n_1}\times \mathrm{Id}_{\Bun_{n_2}})\circ
    \CT_{P,!} $ as desired.
  Let us now analyze $Z$, i.e., the locus of $Y$ where the composition
  \[
    \mathcal{E}\to \mathcal{E}^\prime\to \mathcal{E}^\prime_2
  \]
  is not surjective (over each geometric point). Over each geometric
  point the cokernel of $\mathcal{E}\to \mathcal{E}^\prime$ is of
  length $1$, and supported on one point. Thus,
  \[
    \mathcal{E}^\prime/\mathcal{E} \cong \mathrm{coker}(\mathcal{E}\to
    \mathcal{E}^\prime_2).
  \]
  In particular, $\mathcal{E}^\prime_1$ lies in $\mathcal{E}$ as its
  image in $\mathcal{E}^\prime/\mathcal{E}$ is trivial.  Moreover, the
  quotient
  \[
    \mathcal{E}_2:=\mathcal{E}/\mathcal{E}^\prime_1
  \]
  is again a vector bundle, and the injection
  $\mathcal{E}_1\to \mathcal{E}_1^\prime$ lies in $\Mod^1_{n_2}$.  The
  map
  \[
    Z\to
    \mathrm{Mod}^1_{n_2}\times_{\overrightarrow{h},\Bun_{n_2}}\Bun_P
  \]
  sending the above data to
  $\mathcal{E}_2\hookrightarrow \mathcal{E}^\prime_2$ and
  $0\to \mathcal{E}_1^\prime\to \mathcal{E}^\prime\to
  \mathcal{E}_2^\prime\to 0$ is an isomorphism as $\mathcal{E}$ is
  necessarily the pullback of the diagram
  \[
    \mathcal{E}^\prime\to \mathcal{E}^\prime_2\leftarrow
    \mathcal{E}_2.
  \]
  Now we can argue as in the previous case and deduce that
  \[
    i_\ast i^\ast \mathcal{K}\in D_\et(\Bun_n\times \Bun_M,\Lambda)
  \]
  is the kernel of the functor
  \[
    (\mathrm{Id}_{\Bun_{n_1}}\times \Av_{\mathbb{L}(\frac{n-n_2}{2})[n-n_2],n_2} )\circ
    \CT_{P,!} 
  \]
  as desired. This finishes the proof.
\end{proof}

\begin{remark}
  \label{sec:aver-funct-const-1-strengthening-for-hecke-functor-and-constant-term}
  For $m\geq 1$, let
  \[
    V_{\mathrm{std},m}\in \mathrm{Rep}_\Lambda \hat{G}
  \]
  be the standard representation, with associated Hecke functor
  \[
    T_{\mathrm{std},m}\colon D_\et(\Bun_m,\Lambda)\to D_\et(\Bun_m\times \Div^1,\Lambda).
  \]
  With the same strategy as in the proof of \Cref{sec:averaging-functors-2-averaging-functors-and-constant-terms} one can prove that there exists a canonical exact triangle
  \[
    (T_{V_{\mathrm{std},n_1}}(\frac{n-n_1}{2})[n-n_1]\times \mathrm{Id}_{\Bun_{n_2}})\circ {\mathrm{CT}_{P,!}} \to {\CT_{P,!}}\circ T_{V_{\mathrm{std},n}} \to (\mathrm{Id}_{\Bun_{n_1}}\times T_{V_{\mathrm{std},n_2}}(\frac{n-n_2}{2})[n-n_2])\circ {\mathrm{CT}_{P,!}}
  \]
  of functors
  \[
    D_\et(\Bun_{G},\Lambda)\to D_{\et}(\Bun_M\times \Div^1,\Lambda).
  \]
\end{remark}

As a first consequence, we get that averaging functors preserve the
cuspidal subcategory.

\begin{lemma}
  \label{sec:averaging-functors-3-averaging-functors-cuspidal-subcategory}
  If $\mathcal{F}\in D_{\et,\mathrm{cusp}}(\Bun_n,\Lambda)$ is a cuspidal
  object, then
  \[
    \mathrm{Av}_{\mathbb{L},n}(\mathcal{F})
  \]
  is again cuspidal.
\end{lemma}
\begin{proof}
  This follows directly from
  \Cref{sec:averaging-functors-2-averaging-functors-and-constant-terms}.
\end{proof}

\subsection{Averaging functors and twisting by characters}
\label{sec:aver-funct-twist}

We want to prove a compatibility of averaging functors with twists by characters. In spirit, a similar statement is \cite[Theorem IX.6.1]{fargues2021geometrization}.

Let
\[
  \mathrm{det}\colon \Bun_n\to \Bun_1,\ \mathcal{E}\mapsto \mathrm{det}(\mathcal{E})=\Lambda^n\mathcal{E}
\]
be the determinant morphism.

\begin{proposition}
  \label{sec:aver-funct-twist-1-av-functors-and-twisting}
  Let $\chi \colon W_E\to \Lambda^\times$ be a character, with corresponding $\Lambda$-local system also denoted $\mathbb{L}_{\chi}\in D_\et(\Div^1,\Lambda)$.
  Let $\mathcal{G}\in D_\et(\Bun_1,\Lambda)$ be a Hecke eigensheaf with eigenvalue $\mathbb{L}_{\chi}$. Then for any $\mathcal{F}\in D_\et(\Bun_n,\Lambda)$ there exists a natural isomorphism
  \[
    \mathrm{Av}_{\mathbb{L}}(\mathcal{F}\otimes_\Lambda \mathrm{det}^\ast \mathcal{G})\cong {\mathrm{Av}_{\mathbb{L}\otimes_\Lambda \mathbb{L}_\chi^\vee}}(\mathcal{F})\otimes_\Lambda \mathrm{det}^\ast (\mathcal{G}).
  \]

\end{proposition}
\begin{proof}
  Consider the map
  \[
    f\colon \Bun_n\to \Bun_n\times \Bun_1,\ \mathcal{E}\mapsto (\mathcal{E},\mathrm{det}(\mathcal{E})).
  \]
  Then
  \[
    \mathcal{F}\otimes_\Lambda \mathrm{det}^\ast \mathcal{G}=f^\ast(\mathcal{F}\boxtimes \mathcal{G}).
  \]
   For $r\geq 0$ let
    \[
      \mathrm{Mod}_r^d
    \]
    be the small $v$-stack parametrizing injections $\mathcal{E}\to \mathcal{E}^\prime$ of vector bundles of rank $r$ with cokernel of length $d$ supported at a point of $X$, and let
    \[
      \tilde{f}\colon \mathrm{Mod}_n^1\to \mathrm{Mod}_n^1\times \mathrm{Mod}_1^1
    \]
    be induced by the identity in the first component, and by taking the determinant in the second.
    In the diagram
    \[
      \xymatrix{
        \Bun_n \ar[d]^f & \mathrm{Mod}^1_n\ar[l]_{\overleftarrow{h}}\ar[r]^-{\overrightarrow{h}}\ar[d]^{\tilde{f}} & \Bun_n\times \mathrm{Div}^1\ar[d]^{f\times \Delta} \\
        \Bun_n \times \Bun_1 & \mathrm{Mod}^1_n\times \mathrm{Mod}^1_1 \ar[l]_-{\overleftarrow{k}}\ar[r]^-{\overrightarrow{k}} & \Bun_n\times \Bun_1\times \mathrm{Div}^1\times \mathrm{Div}^1 
        }
      \]
      both squares are cartesian\footnote{Note that in this proof $\overrightarrow{h}$ has target $\Bun_n\times \Div^1$, not $\Bun_n$ as in previous occasions. We hope this does not create ever lasting confusion.}. Indeed, given a modification $\mathcal{E}\hookrightarrow \mathcal{E}^\prime\in \Mod^1_n$ with cokernel of length $1$ supported at a point $x\in X$, and a modification $\mathcal{L}\hookrightarrow \mathrm{det}(\mathcal{E}^\prime)$ of length $1$ supported on the same point, then there exists a unique isomorphism
      \[
        \mathcal{L}\cong \mathrm{det}(\mathcal{E})
      \]
      compatible with the morphisms to $\mathrm{det}(\mathcal{E}^\prime)$ as both are canonically isomorphic to $\mathrm{det}(\mathcal{E}^\prime)(-x)$. The argument for the left square is similar (however, we only need that the right square is cartesian).
      Let $V_{\mathrm{std},r}$ be the standard representation of $\GL_{r,\Lambda}$.
      Recall that
      \[
        T_{V_{\mathrm{std},n}}(\mathcal{F})=\overrightarrow{h}_!(\overleftarrow{h}^\ast(\mathcal{F})((n-1)/2)[n-1]),
      \]
      i.e., $T_{V_{\mathrm{std},n}}$ is defined as the correspondence given by the top horizontal line (with kernel $\Lambda((n-1)/2)[n-1]$).
      Using the proven cartesianess, proper base change and the projection formula we can deduce that
      \[
        \begin{matrix}
          & T_{V_{\mathrm{std},n}}(\mathcal{F}\otimes \mathrm{det}^\ast \mathcal{G}) \\
          \cong & T_{V_{\mathrm{std},n}}(f^\ast(\mathcal{F}\boxtimes \mathcal{G})) \\
          \cong & (f\times \Delta)^\ast(\overrightarrow{k}_!(\overleftarrow{k}^\ast(\mathcal{F}\boxtimes \mathcal{G})((n-1)/2)[n-1]) \\
          \cong & (f\times \Delta)^\ast(T_{V_{\mathrm{std},n}}(\mathcal{F})\boxtimes T_{V_{\mathrm{std},1}}(\mathcal{G})) \\
          \end{matrix}
        \]
        as $\mathcal{G}$ is a Hecke eigensheaf with eigenvalue $\mathbb{L}_{\chi}$ the last object can be rewritten as
        \[
          (f\times \Delta)^\ast(T_{V_{\mathrm{std},n}}(\mathcal{F})\boxtimes \mathbb{L}_{\chi}^\vee \boxtimes \mathcal{G}).
        \]
        All this isomorphisms are $W_{E}$-equivariant if $W_E$ acts diagonally on the exterior powers. Let $p_2\colon \Bun_n \times \mathrm{Div}^1 \to \Div^1$ be the second projection. From here we can conclude  
        \[
        \begin{matrix}
                      \mathrm{Av}_{\mathbb{L}}(\mathcal{F}\otimes \mathrm{det}^\ast \mathcal{G}) & \cong &   p_{2,!} (\mathbb{L} \otimes T_{V_{\mathrm{std},n}}(\mathcal{F}\otimes \mathrm{det}^\ast \mathcal{G})(1-n)[1]) \\
                     & \cong  & p_{2,!}(\mathbb{L}\otimes_{\Lambda} \mathbb{L}_{\chi}^\vee \otimes_\Lambda T_{V_{\mathrm{std},n}}(\mathcal{F})(1-n)[1])\otimes _{\Lambda} \mathrm{det}^\ast \mathcal{G} \\
       &   \cong  & \mathrm{Av}_{\mathbb{L}\otimes_\Lambda \mathbb{L}_\chi^\vee}(\mathcal{F})\otimes_\Lambda \mathrm{det}^\ast (\mathcal{G}),
       \end{matrix}              
        \]
        as desired.
\end{proof}

\subsection{The averaging functors of rank $1$ local systems}
\label{sec:aver-funct-triv}

  The categorical version of Fargues' conjecture suggests that if $\mathbb{L}$ is a $\Lambda$-local system of rank $<n$, then $\mathrm{Av}_{\mathbb{L}}$ sends each object in
  \[
    D_{\et,\mathrm{cusp}}(\Bun_n,\Lambda)
  \]
  to zero. We want to prove this in the case $\mathbb{L}$ has rank $1$, and $n=2$.

  \begin{theorem}
    \label{sec:aver-funct-triv-1-av-functor-vanishes-for-n-2-e-trivial}
    Let $\mathbb{L}_\chi\colon W_E\to \Lambda^\times$ be a character, seen as a local system $\mathbb{L}_\chi$ in $\Div^1$.
   For each $\mathcal{F}\in D_{\et,\mathrm{cusp}}(\Bun_2,\Lambda)$,
    \[
      \mathrm{Av}_{\mathbb{L}_\chi}(\mathcal{F})=0
    \]
  \end{theorem}
  If $n=1$, then each object is cuspidal for our definition (as there are no proper parabolics) and the statement is wrong.
    The proof follows \cite[Theorem A.4]{gaitsgory_on_a_vanishing_conjecture}.
    \begin{proof}
      Most of the proof works for any $n\geq 2$, so we write it in this generality, only specifying $n=2$ when needed (although we expect the same argument to extend to all $n\geq 2$ with more work). By \Cref{sec:aver-funct-twist-1-av-functors-and-twisting} and \Cref{sec:const-terms-twist-1-constant-terms-and-twisting} we can assume that $\mathbb{L}_\chi\cong \Lambda$ is trivial.  Set
      \[
        \mathrm{Av}:=\mathrm{Av}_{\Lambda}.
      \]
    Recall that each cuspidal $\mathcal{F}$ is automatically supported on the semistable locus of $\Bun_n$, cf.\ \Cref{sec:v_m-1-implies-cuspidal-implies-support-on-non-semistable-locus}, and that $\mathrm{Av}(\mathcal{F})\in D_{\et,\mathrm{cusp}}(\Bun_n,\Lambda)$ is again cuspidal, cf.\ \Cref{sec:averaging-functors-3-averaging-functors-cuspidal-subcategory}. Thus it suffices to prove that for each $b,c\in B(G)$ basic with $\kappa(c)=\kappa(b)+1$, and each $\mathcal{F}\in D_{\et,\mathrm{cusp}}(\Bun_n,\Lambda)$ which is supported on $\Bun_n^b$, the stalk
    \[
      j_c^\ast\mathrm{Av}(\mathcal{F})
    \]
    at $c$ of
    \[
      \mathrm{Av}(\mathcal{F})
    \]
    vanishes.
    To prove this, let
    \[
      Z
    \]
    be the small $v$-stack over $k$ parametrizing morphisms
    \[
      \mathcal{E}\to \mathcal{E}^\prime
    \]
    with $\mathcal{E}$ locally isomorphic to $\mathcal{E}_b$ and $\mathcal{E}^\prime$ locally isomorphic to $c$.
    For each $r=0,\ldots, n$ and each $d\in B(\GL_r)$ let
    \[
      Z_{r,d}
    \]
    be the small $v$-stack parametrizing factorizations into a surjection followed by an injection
    \[
      \mathcal{E}\twoheadrightarrow \mathcal{F}\hookrightarrow \mathcal{E}^\prime
    \]
    with $\mathcal{E}, \mathcal{E}^\prime$ locally isomorphic to $\mathcal{E}_b, \mathcal{E}_c$ respectively, and $\mathcal{F}$ locally isomorphic to $\mathcal{E}_d$. The morphism
    \[
      (\mathcal{E}\twoheadrightarrow \mathcal{F}\hookrightarrow \mathcal{E}^\prime)\mapsto (\mathcal{E}\to \mathcal{E}^\prime)
    \]
    defines an injection $j_{r,d}\colon Z_{r,d}\to Z$, which is a locally closed embedding. Moreover,
    \[
      \coprod\limits_{r,d} |Z_{r,d}|=|Z|
    \]
    as each geometric point of $Z$ factors over a unique $Z_{r,d}$.
    The space $Z$ defines a correspondence
    \[
      \xymatrix{
        & Z\ar[ld]_{q_1}\ar[rd]^{q_2}& \\
        \Bun_n & & \Bun_n 
      }      
    \]
    and thus for each $r,d$ an endofunctor
    \[
      \mathrm{Av}_{r,d}\colon D_\et(\Bun_n,\Lambda)\to D_\et(\Bun_n,\Lambda),\ \mathcal{G}\mapsto q_{2,!}(q_1^{\ast}\mathcal{G}\otimes^{\mathbb{L}}_\Lambda j_{r,d,!}(\Lambda))
    \]
    as well as the endofunctor
    \[
      \overline{\mathrm{Av}}\colon D_\et(\Bun_n,\Lambda)\to D_\et(\Bun_n,\Lambda),\ \mathcal{G}\mapsto q_{2,!}(q_1^{\ast}\mathcal{G})).
    \]
    Note that (up to a shift and twist) $\mathrm{Av}_{n,b}=\mathrm{Av}$, and $Z_{n,d}$ is empty for $d\neq b$.
    We claim that for $\mathcal{F}\in D_{\et,\mathrm{cusp}}(\Bun_n,\Lambda)$ each of the objects
    \[
      \overline{\mathrm{Av}}(\mathcal{F}),\quad \mathrm{Av}_{r,d}(\mathcal{F}),
    \]
    where $r= 0,\ldots, n-1$ and $d$ is arbitrary, vanish at $c$. As the cone of the natural morphism
    \[
      \mathrm{Av}(\mathcal{F})\to \overline{\mathrm{Av}}(\mathcal{F})
    \]
    has a filtration by the $\mathrm{Av}_{r,d}(\mathcal{F})$ for $r=0,\ldots,n-1$, this implies the desired vanishing of
    \[
      \mathrm{Av}_{n,b}(\mathcal{F})=\mathrm{Av}(\mathcal{F}).
    \]
    For $\overline{\mathrm{Av}}$ note that the base change along $k \to \Bun_n^c$ of the projection $Z\to \Bun_n^c$ 
    \[
      Z\times_{\Bun_n^c}k\cong [G_b({E})\setminus \mathrm{Hom}(\mathcal{E}_b, \mathcal{E}_c)]
    \]
    where $\Hom(\mathcal{E}_b,\mathcal{E}_c)$ is a \textit{positive} Banach-Colmez space (by our assumption that $b,c$ are basic with $\kappa(c)>\kappa(b)$).
    As the cohomology of positive Banach-Colmez spaces is $1$-dimensional and concentrated in a single degree the compactly supported cohomology of
    \[
      Z
    \]
    with coefficients in the pullback of a (geometrically) cuspidal object in $D_\et(\Bun_n^b,\Lambda)\cong D_\et([\ast/G_b({E})],\Lambda)$ vanishes for $n=2$: when $\kappa(b)$ is even, cf. the discussion of \Cref{sec:geom-cusp-repr-geom-cusp-representations-for-n-2}; when $\kappa(b)$ is odd, $G_b(E)$ is compact modulo center, so the claim is easy.
    Next assume $r=0$, then
    \[
      Z_{r,d}\times_{\Bun_n,c}k=[\ast/G_b({E})]
    \]
    and the claim follows as geometrically cuspidal objects don't have cohomology, exactly as above. For $r=1,\ldots, n-1$ and $d\in B(\GL_r)$ we can write
    \[
      Z_{r,d}\times_{\Bun_n^c}k=[G_b({E})\setminus \mathrm{Hom}^{\mathrm{surj}}(\mathcal{E}_b,\mathcal{E}_d)\times^{\mathcal{G}_b} \mathrm{Hom}^{\mathrm{inj}}(\mathcal{E}_d, \mathcal{E}_c)]
    \]
    with $\mathcal{G}_b$ acting diagonally on $\mathrm{Hom}^{\mathrm{surj}}(\mathcal{E}_b,\mathcal{E}_d)\times \mathrm{Hom}^{\mathrm{inj}}(\mathcal{E}_d, \mathcal{E}_c)$.
    Let $\mathcal{F}\in D_{\mathrm{cusp}}(\Bun_n,\Lambda)$ with support at $b$.
    The compactly supported cohomology 
    \[
      R\Gamma_c([G_b({E})\setminus \mathrm{Hom}^{\mathrm{surj}}(\mathcal{E}_b,\mathcal{E}_d)],\mathcal{F})
    \]
    vanishes because it identifies with the fiber of $0=Rf_!(\mathrm{CT}_{P,!}(\mathcal{F}))$ at $\mathcal{E}_d$, where $f$ is the morphism
    \[
      \Bun_M\xrightarrow{f} \Bun_r,
    \]
    and $P\subseteq G$ is the standard parabolic with Levi $M\cong \GL_r\times \GL_{n-r}$.
    Using the projection
    \[
      [G_b({E})\setminus \mathrm{Hom}^{\mathrm{surj}}(\mathcal{E}_b,\mathcal{E}_d)\times^{\mathcal{G}_b} \mathrm{Hom}^{\mathrm{inj}}(\mathcal{E}_d, \mathcal{E}_c)]\to [\mathcal{G}_b\setminus \mathrm{Hom}^{\mathrm{inj}}(\mathcal{E}_d,\mathcal{E}_d)]
    \]
    with fibers isomorphic to $[G_b({E})\setminus \mathrm{Hom}^{\mathrm{surj}}(\mathcal{E}_b,\mathcal{E}_d)]$ this implies the desired vanishing of the stalk of
    \[
      \mathrm{Av}_{r,d}(\mathcal{F})
    \]
    at $c$, i.e., of the compactly supported cohomology of $Z_{r,d}\times_{\Bun_n^c}k$ with coefficients in $\mathcal{F}$.
  \end{proof}

\subsection{Averaging functors via the spectral action}
\label{sec:aver-funct-via}

Recall the stack $X_{\hat{G}}$ of $n$-dimensional continuous $\Lambda$-representations of $W_E$, with the natural projection
\[
  f\colon X_{\hat{G}}\to [\mathrm{Spec}(\Lambda)/\hat{G}],
\]
introduced in \Cref{sec:work-of-fargues-scholze}. If $V \in \mathrm{Rep}_{\hat{G}}$, we also denote by $V$ the associated vector bundle on $[\mathrm{Spec}(\Lambda)/\hat{G}]$.

For a $W_E$-representation $\mathbb{L}$ on a finite projective $\Lambda$-module, we set
\[
 \mathcal{A}v_{\mathbb{L}}:=R\Gamma(W_E,\mathbb{L}\otimes_\Lambda f^\ast V_{\mathrm{std}})\in \mathrm{Perf}(X_{\hat{G}}). 
\]

Note that $\mathcal{A}v_{\mathbb{L}}$ is indeed a perfect complex. Indeed, the wild inertia $P_E\subseteq W_E$ acts via a finite quotient, and has no higher cohomology. As $p\in \Lambda^\times$, this implies that the invariants
\[
 (f^\ast V_{\mathrm{std}})^{P_E} 
\]
are a direct summand of $f^\ast V_{\mathrm{std}}$ and hence a vector bundle.
After fixing a generator $\tau \in I_E/P_E$ of the tame quotient of the inertia subgroup $I_E\subseteq W_E$, and a Frobenius lift $\sigma\in W_{E}/P_E$ the complex $\mathcal{A}v_{\mathbb{L}}$ is calculated as the homotopy limit of the diagram 
\begin{equation}
  \label{equation-diagram-for-w-e-invariants}
  \xymatrix{
 (\mathbb{L}\otimes_{\Lambda}f^\ast V_{\mathrm{std}})^{P_E} \ar[r]^{\tau-1}\ar[d]^{\sigma-1} & (\mathbb{L}\otimes_{\Lambda}f^\ast V_{\mathrm{std}})^{P_E}\ar[d]^{\sigma(1+\tau+\ldots+\tau^{q-1})-1} \\
    (\mathbb{L}\otimes_{\Lambda}f^\ast V_{\mathrm{std}})^{P_E} \ar[r]^{\tau-1} & (\mathbb{L}\otimes_{\Lambda}f^\ast V_{\mathrm{std}})^{P_E},   
}
\end{equation}
where $q$ is the cardinality of the residue field of $E$. This presentation implies that $\mathcal{A}v_{\mathbb{L}}$ is a perfect complex on $X_{\hat{G}}$ as claimed.

\begin{lemma}
  \label{sec:aver-funct-via-1-averaging-functor-via-spectral-action}
  There exists a canonical isomorphism
  \[
   \mathrm{Av}_{\mathbb{L}}\cong \mathcal{A}v_{\mathbb{L}}\ast (-)
  \]
  of functors
  \[
    D_{\et}(\Bun_n,\Lambda)\to D_{\et}(\Bun_n,\Lambda).
  \]
\end{lemma}
\begin{proof}
  Note that the diagram (\Cref{equation-diagram-for-w-e-invariants}) is a diagram of perfect complexes in $X_{\hat{G}}$. Applying it to some object $\mathcal{G}\in D_{\et}(\Bun_n,\Lambda)$ we get the diagram 
\[
  \xymatrix{
   (\mathbb{L}\otimes_\Lambda T_{V_{\mathrm{std}}}(\mathcal{G}))^{P_E}  \ar[r]^{\tau-1}\ar[d]^{\sigma-1} & (\mathbb{L}\otimes_\Lambda T_{V_{\mathrm{std}}}(\mathcal{G}))^{P_E}\ar[d]^{\sigma(1+\tau+\ldots+\tau^{q-1})-1} \\
    (\mathbb{L}\otimes_\Lambda T_{V_{\mathrm{std}}}(\mathcal{G}))^{P_E} \ar[r]^{\tau-1} & (\mathbb{L}\otimes_\Lambda T_{V_{\mathrm{std}}}(\mathcal{G}))^{P_E}, 
  }
  \]
  whose homotopy limit calculates $\mathrm{Av}_{\mathbb{L}}$. But the spectral action 
  \[
    \mathrm{Perf}(X_{\hat{G}})\times D_{\et}(\Bun_n,\Lambda)\to D_{\et}(\Bun_n,\Lambda)
  \]
  commutes with finite homotopy limits in each variable. This implies the claim.
\end{proof}

The following immediate consequence of the existence of the spectral
action is the analog of the so-called vanishing conjecture of
\cite{frenkel_gaitsgory_vilonen}, cf. also \cite{gaitsgorygeneralized} for a generalized vanishing conjecture in the context of
$D$-modules.

For clarity let us add the subscript
\[
  \mathrm{Av}_{\mathbb{L},n},
\]
when we consider the averaging functor on $D_{\et}(\Bun_n,\Lambda)$ associated with $\mathbb{L}$. Note that $n$ need not be the rank of $\Lambda$.
\begin{proposition}
  \label{sec:aver-funct-via-1-vanishing-conjecture}
  Assume that $\Lambda$ is a field and $\mathbb{L}$ is an irreducible $W_E$-representation. Then
  \[
    \mathrm{Av}_{\mathbb{L},n}=0
  \]
  if $n<\mathrm{rk}(\mathbb{L})$.
\end{proposition}
\begin{proof}
  By \Cref{sec:aver-funct-via-1-averaging-functor-via-spectral-action} it suffices to see that the object
  \[
    \mathcal{A}v_{\mathbb{L}}\in \mathcal{X}_{\hat{G}}
  \]
  vanishes. This can be checked at field valued points of $X_{\hat{G}}$, where it follows from the fact that by irreducibility of $\mathbb{L}$
  \[
    R\Gamma(W_E,\mathbb{L}\otimes_\Lambda V)=0
  \]
  for any representation $V$ of $W_E$ which is of rank $n<\textrm{rank } \mathbb{L}$. Indeed, vanishing in degree $0$ is clear, which implies vanishing in degree $2$ by duality, and then vanishing of the $H^1$ as the Euler characteristic of $W_E$ on finite dimensional $\Lambda$-representations is $0$. 
\end{proof}

This vanishing has the following interesting consequence.

\begin{corollary}
  \label{sec:aver-funct-via-2-hecke-eigensheaves-of-irred-autom-cuspidal}
  Assume that $\Lambda$ is a field.
  Let $\mathcal{F}\in D_\et(\Bun_n,\Lambda)$ be a Hecke eigensheaf
  with eigenvalue $\mathbb{L}$ an irreducible $W_{E}$-representation over $\Lambda$
  of rank $n$. Then
  \[
    \mathcal{F}\in D_{\et,\mathrm{cusp}}(\Bun_n,\Lambda),
  \]
  i.e., $\mathcal{F}$ is cuspidal.
\end{corollary}
\begin{proof}
  From \Cref{sec:aver-funct-via-1-vanishing-conjecture} and
  \Cref{sec:averaging-functors-2-averaging-functors-and-constant-terms}
  we can conclude that
  \[
    \CT_{P,!}\circ \Av_{\mathbb{L},n}(\mathcal{F})=0
  \]
  for any standard parabolic $P\subsetneq G$. As $\mathcal{F}$ is
  a Hecke eigensheaf with eigenvalue $\mathbb{L}$, we get that
  \[
    \mathrm{Av}_{\mathbb{L},n}(\mathcal{F})=R\Gamma(W_{E},\mathbb{L}^\vee{\otimes}_\Lambda
    \mathbb{L})\otimes_{\Lambda}^{L}\mathcal{F}.
  \]
  Thus, $\CT_{P,!}(\mathcal{F})=0$ as
  \[
    0=\CT_{P,!}({\Av_{\mathbb{L},n}}(\mathcal{F}))=R\Gamma(W_{E},\mathbb{L}^\vee{\otimes}_\Lambda
    \mathbb{L})\otimes_{\Lambda}^{L}\CT_{P,!}(\mathcal{F})
  \]
  and
  \[
    R\Gamma(W_{E},\mathbb{L}^\vee{\otimes}_\Lambda
    \mathbb{L})\neq 0.
  \]
  This finishes the proof.
\end{proof}

\subsection{Averaging functors on $D_{\lis}$}
\label{sec:aver-funct-d_lis-1}

We explain how to extend the definition of the averaging functors
\[
  \mathrm{Av}_{\mathbb{L}}\colon D_\et(\Bun_n,\Lambda)\to D_\et(\Bun_n,\Lambda)
\]
for $\mathbb{L}\in D_{\et}(\mathrm{Div}^1,\Lambda)$, to the case that the torsion ring $\Lambda$ is replaced by an arbitrary $\Z_\ell$-algebra $L$ (the most important case being $L=\Qlbar$).

Let $\mathbb{L}$ be a representation of $W_E$ on a finite, locally free $L$-module $\mathbb{L}$, which is continuous when $\mathbb{L}$ is given the colimit topology in the presentation
\[
  \mathbb{L}=\varinjlim\limits_{N\subseteq \mathbb{L} \textrm{ fin.gen. } \Z_\ell-\textrm{submodule}} N,
\]
with each $N$ equipped with the $\ell$-adic topology.
Then for each $\mathcal{F}\in D_\lis(\Bun_n,L)$ the object
\[
  {\mathbb{L}} \overset{\bs}\otimes_{L}^L \mathcal{F}\in D_{\bs}(\Bun_n, L) 
\]
lies in $D_\lis(\Bun_n,L)$. Using the presentation (\Cref{equation-diagram-for-w-e-invariants}) we can therefore deduce that
\[
  R\Gamma(W_E,\mathbb{L} \overset{\bs}\otimes^{L}_L \mathcal{F})\in D_\lis(\Bun_n,L)
\]
for each $W_E$-equivariant object $\mathcal{F}\in D_\lis(\Bun_n,L)$.

The functor
$$
T_{V_{\rm std}}: D_\bs(\Bun_n, L) \to D_\bs(\Bun_n \times \mathrm{Div}^1, L), 
$$
defined by the formula
$$
T_{V_{\rm std}} (\mathcal{F}) = \overrightarrow{h}_{\natural} \left( \overleftarrow{h}^\ast \mathcal{F} \overset{\bs} \otimes_{L}^L \mathcal{S}_{V_{\rm std}}^\prime \right),
$$
where $\mathcal{S}_{V_{\rm std}}^\prime$ denotes the object constructed in \cite[\S IX.2]{fargues2021geometrization}, induces, by \cite[Corollary IX.2.3]{fargues2021geometrization} a functor
$$
T_{V_{\rm std},\lis}^{\rm equiv}: D_\lis(\Bun_n, L) \to D_\lis(\Bun_n, L)^{BW_E}.
$$
By the considerations above, if $\mathbb{L}$ is a representation of $W_E$ as before, for any $\mathcal{F} \in D_\lis(\Bun_n, L)$, we have
\[
  \mathrm{Av}_{\mathbb{L}}(\mathcal{F}):=R\Gamma(W_E,\mathbb{L} \overset{\bs}\otimes^{L}_L T_{V_{\rm std},\lis}^{\rm equiv}(\mathcal{F}))\in D_\lis(\Bun_n,L).
\]
This recipe defines the \textit{averaging functor} associated to $\mathbb{L}$
\[
  \mathrm{Av}_{\mathbb{L}}\colon D_\lis(\Bun_n,L)\to D_\lis(\Bun_n,L)
\]
It recovers the previous definition when $L=\Lambda$ is torsion, through the identification $D_\lis(\Bun_n,L)\cong D_\et(\Bun_n,\Lambda)$, thanks to \cite[Proposition VII.5.2]{fargues2021geometrization}.

The description 
\[
  \mathrm{Av}_{\mathbb{L}}(-)\cong \mathcal{A}v_{\mathbb{L}}\ast(-)
\]
of it via the spectral action for the analogously defined perfect complex
\[
  \mathcal{A}v_{\mathbb{L}}:=R\Gamma(W_E,\mathbb{L}{\otimes}_L f^\ast V_{\mathrm{std}})
\]
on the algebraic stack $X_{\hat{G}}$ over $L$, remains valid.

\begin{remark}
\label{commutation-ct-averaging-d-lis}
We do not know how to prove an analogue of \Cref{sec:averaging-functors-2-averaging-functors-and-constant-terms} on $D_\lis(-,L)$ when $\mathrm{char} L=0$ since we don't know if excision holds in this case.
\end{remark}

\section{Application to irreducibility of the Fargues-Scholze parameters if $n=2$}
\label{sec:irred-farg-scholze}

We keep the assumptions from \Cref{sec:notation}, but assume additionally that $\Lambda$ is an algebraically closed field of characteristic $\ell$.
Moreover, we assume $n=2$, i.e., 
\[
  G=\GL_2.
\]
Let $b\in B(G)$ be basic and $\pi\in \mathrm{Rep}^\infty_\Lambda G_b(E)$ be irreducible geometrically cuspidal. We want to prove that the $L$-parameter
\[
  \varphi_\pi\colon W_E\to \hat{G}(\Lambda)
\]
associated to $\pi$ by Fargues-Scholze is irreducible, as claimed in the introduction, \Cref{sec:introduction-1-main-theorem-on-irreducibility}.

\begin{theorem}
  \label{sec:appl-irred-farg-on-cusp-hecke-action-contains-no-rank-1-stuff}
  Assume that $\mathcal{F}\in D_{{\et,\mathrm{cusp}}}(\Bun_n,\Lambda)$ is ULA. Then any irreducible subquotient of any cohomology sheaf of the $W_E$-equivariant object 
  \[
    T_{V_{\mathrm{std}}}(\mathcal{F})
  \]
  is of rank $2$.
  
     In particular, the Fargues-Scholze parameter associated to a smooth irreducible supercuspidal representation $\pi$ of $\GL_2(E)$ is irreducible.
\end{theorem}
We note that by \Cref{sec:const-term-funct-prop-geom-cuspi-implies-supercuspidal} the category of
\[
  D_{\et,\mathrm{cusp}}(\Bun_n,\Lambda)\cong \prod\limits_{b\in B(G) \textrm{ basic}} D_{\et,\mathrm{cusp}}(\Bun_n^b,\Lambda)
\]
has a natural $t$-structure. Thus the statement of \Cref{sec:appl-irred-farg-on-cusp-hecke-action-contains-no-rank-1-stuff} makes sense.

\begin{proof}
Since $\mathcal{F}\in D_{{\et,\mathrm{cusp}}}(\Bun_n,\Lambda)$, by \Cref{sec:v_m-1-implies-cuspidal-implies-support-on-non-semistable-locus} and \Cref{sec:const-term-funct-prop-geom-cuspi-implies-supercuspidal}, $\mathcal{F}$ is supported on the semistable locus and represented by cuspidal representations for varying $G_b(E)$ there. 
  
  By \Cref{sec:aver-funct-const-1-strengthening-for-hecke-functor-and-constant-term} the object
  \[
    T_{V_{\mathrm{std}}}(\mathcal{F})
  \]
  is also supported on the semistable locus and represented by cuspidal representations for varying $G_b(E)$ there.
  By \cite[Theorem I.5.1]{fargues2021geometrization}, the stalks are admissible. Thus, as a module over the Bernstein center, the object $T_{V_{\mathrm{std}}}(\mathcal{F})$ can be represented as a complex of modules of finite length. Forgetting the action of the Bernstein center, the first statement then follows from \Cref{sec:appl-irred-farg-complex-of-lambda-modules-with-vanishing-averaging-functors} below because of \Cref{sec:aver-funct-triv-1-av-functor-vanishes-for-n-2-e-trivial} and \Cref{sec:aver-funct-via-1-vanishing-conjecture}.
  If $\ell\neq 2$, the last statement follows because if $\pi$ is a smooth irreducible supercuspidal representation of $\GL_2(E)$, $j_{1,!}\mathcal{F}_\pi$ is cuspidal, cf. \Cref{sec:geom-cusp-repr-geom-cusp-representations-for-n-2}, and because for $\ell \neq 2$, irreducible representations of dimension $1$ or $2$ are uniquely determined by their traces, the traces are linearly independent and the excursion operators associated with the excursion data
  \[
   V_{\mathrm{std}}^\ast, V_{\mathrm{std}}, (1,\gamma)\in W_E^{2}, \Lambda\hookrightarrow V^{\ast}_{\mathrm{std}}\otimes_\Lambda V_{\mathrm{std}}\twoheadrightarrow \Lambda
 \]
 records the trace of $\gamma$.

 Thus we are left with the case that $\ell=2$, where we have to argue differently. The following proof works for arbitrary $\ell$. Let $A=\Zlbar$,  $L=\Qlbar$. Since Fargues-Scholze's construction of $L$-parameters is compatible with torsion by characters, we can always twist $\pi$ by a character. Therefore, cf. \cite[Theorem 4.17, Corollary 4.20]{henniart_vigneras_representations}, we are free to assume that $\Lambda=\Flbar$.  
 In this case, by \cite[Th\'eor\`eme 3.26]{minguez_secherre_types_modulo_l}, if we set $\mathcal{F}=j_{1,!} \mathcal{F}_\pi \in D_\et(\Bun_2, \Lambda)$, we can find some $\mathcal{F}_A\in D_\lis(\Bun_2,A)$ with the same support, which corresponds to some shift of a sheaf associated with an admissible representation on an $\ell$-adically separated, $\ell$-torsion free $A$-module, such that $\mathcal{F}_A\otimes_A^L\Lambda\cong \mathcal{F}$. Set $\mathcal{F}_L:=\mathcal{F}_A\otimes_A^L L$.
 Then $\mathcal{F}_L$ is again Schur-irreducible and by what we have shown about $T_{V_{\mathrm{std}}}(\mathcal{F})$, its associated Fargues-Scholze parameter has the property that its reduction to $\Lambda$ is irreducible. Indeed, as $L$ has characteristic $0$ the $L$-parameter $\varphi_{\mathcal{F}_L}$ of $\mathcal{F}_L$ is uniquely determined by its traces while $T_{V_{\mathrm{std}}}(\mathcal{F}_L)$ reduces $W_E$-equivariantly to $T_{V_{\mathrm{std}}}(\mathcal{F})$.  Let $e\in H^0(X_{\hat{\GL_2},A},\mathcal{O})$ be the idempotent cutting out the locus of $L$-parameters admitting a model over $A$ whose reduction is irreducible. Then $e$ acts trivially on $\mathcal{F}_L$. Because $\mathrm{End}_{D_\lis(\Bun_2,A)}(\mathcal{F}_A)$ is $\ell$-torsion free we can conclude that $e$ acts trivially on $\mathcal{F}_A$. Thus, $e$ acts trivially in $\mathcal{F}$ and we can conclude that the $L$-parameter associated to $\mathcal{F}$ must be irreducible. This finishes the proof.  
\end{proof}

\begin{lemma}
  \label{sec:appl-irred-farg-complex-of-lambda-modules-with-vanishing-averaging-functors}
  Let $A\in D(\mathrm{Rep}_\Lambda^{\mathrm{cont}} W_E)$ be a bounded complex with finite dimensional cohomology objects. Fix $d\geq 0$. Assume that
  \[
    R\Gamma(W_E,\mathbb{L}\otimes^{L}_\Lambda A)=0
  \]
  for each irreducible $W_{E}$-representation $\mathbb{L}$ of rank $\neq d$.
  Then each irreducible subquotient of each cohomology sheaf of $A$ is of rank $d$. 
\end{lemma}
\begin{proof}
  We may assume that $H^{i}(A)=0$ for $i<0$ and $H^0(A)\neq 0$.
  Assume that $V\subseteq H^0(A)$ is an irreducible $W_E$-subrepresentation, and define $B$ via the triangle
  \[
    V\to A\to B.
  \]
  Then the morphism
  \[
    R\Gamma(W_E,V^\ast\otimes_\Lambda V)\to R\Gamma(W_E,V^\ast\otimes_\Lambda A)
  \]
  is an injection of a non-zero object in degree $0$.  By our assumption we can conclude that $V$ is of rank $d$.
  Let $W$ be an irreducible $W_E$-representation of rank $\neq d$. Then
  \[
    R\Gamma(W_E,W\otimes_\Lambda V)=0.
  \]
  Indeed, in degree $0$ this follows as $V,W$ are irreducible of different rank $2$, by duality one deduces similarly the same in degree $2$, and then vanishing in degree $1$ follows as the Euler characteristic of $W_E$-cohomology is $0$. 
  We can therefore deduce that $B$ satisfies the same assumption as $A$. Hence, we may induct on the number of irreducible subquotients of the cohomology sheaves of $A$. This finishes the proof.  
\end{proof}


\section{Application to Fargues' conjecture in the irreducible case}
\label{sec:proof-fargues-conjecture-irreducible-case}

We want to give some application of the averaging functors to the irreducible case of Fargues' conjecture for $\GL_n$.
Namely, starting with an irreducible $W_E$-representation $\mathbb{L}$ over $L=\Qlbar$ we want to construct a non-zero irreducible Hecke eigensheaf with eigenvalue $\mathbb{L}$. However, to construct it \textit{we have to assume the local Langlands correspondence for $\GL_n$,} or more precisely its realization in the compactly supported cohomology of the Lubin-Tate tower.

We keep the notations from \Cref{sec:notation}, but add the assumption that $L$ is an algebraically closed field of characteristic $0$, e.g., $L=\Qlbar$.
When speaking about irreducible representations of $W_E$ on $L$-modules, we implicitly assume that they are continuous when $L$ is equipped with the colimit topology of the $\ell$-topology of its finitely generated $\Z_\ell$-submodules.

\subsection{Generalities on Hecke eigensheaves for irreducible representations}
\label{sec:gener-hecke-eigensh}

Let $\mathbb{L}$ be an \textit{irreducible} $L$-representation of $W_E$.
Let
\[
  \mathcal{F}\in D_\lis(\Bun_n,L)
\]
be a Hecke eigensheaf with eigenvalue $\mathbb{L}$, i.e., we are given natural isomorphisms
\[
  T_V(\mathcal{F})\cong r_{V,\ast}(\mathbb{L})\boxtimes_L \mathcal{F}
\]
for each representation $V\in \mathrm{Rep}_L(\hat{G})$, cf. \Cref{sec:statement-of-fargues-conjecture}.
 We moreover assume that $\mathcal{F}$ is ULA.

Let
\[
  \mathrm{LL}_{FS}\colon \mathcal{O}(X_{\hat{G},L})\to \mathcal{Z}(G,L)
\]
be the morphism into the Bernstein center $\mathcal{Z}(G,L)$ for $D_{\lis}(\Bun_n,L)$ constructed by Fargues-Scholze in \cite{fargues2021geometrization}, and let 
\[
  V_1,\ldots, V_m\in \mathrm{Rep}_L (\hat{G}), (\gamma_1,\ldots, \gamma_m)\in W_E^{m}, \alpha\colon L \to V_1\otimes_L\ldots, \otimes_L V_m, \beta\colon V_1\otimes_L\ldots, \otimes_L V_m\to L 
\]
be an excursion datum with associated function $f\in \mathcal{O}(X_{\hat{G},L})$. By the Hecke eigensheaf property we can then see that $\mathrm{LL}_{\mathrm{FS}}(f)$ acts
on $\mathcal{F}$ via the same scalar $\in L$, which gives the composition
\[
  L\xrightarrow{\alpha} r_{V_1,\ast}(\mathbb{L})\otimes_L \ldots \otimes_L r_{V_m,\ast}(\mathbb{L}) \xrightarrow{\gamma_1\otimes\ldots \otimes \gamma_m} r_{V_1,\ast}(\mathbb{L})\otimes_L \ldots \otimes_L r_{V_m,\ast}(\mathbb{L}) \xrightarrow{\beta} L.
\]
By compatibility with parabolic induction, cf.\ \cite[Theorem IX.6.1]{fargues_geometrization_of_the_local_langlands_correspondence_overview}, this implies the following lemma.

\begin{lemma}
  \label{sec:gener-hecke-eigensh-1-hecke-eigensheaf-supported-on-semistable-locus}
  Let as before $\mathcal{F}\in D_\lis(\Bun_n,L)$ be a ULA Hecke eigensheaf with (irreducible) eigenvalue $\mathbb{L}$. Then $\mathcal{F}$ is supported on the semistable locus and supercuspidal there.
\end{lemma}
\begin{proof}
  Let $e\in \mathcal{O}(X_{\hat{G}})$ be the idempotent cutting out the connected component of $\mathbb{L}\in X_{\hat{G}}(L)$. By \cite[Chapter VIII]{fargues2021geometrization} the element $e$ can be represented by an excursion datum like above. By \cite[Theorem IX.6.1]{fargues2021geometrization} the element $e$ must act trivially on every ULA sheaf supported on the non-semistable, or on any parabolic induction supported on the semistable locus.
\end{proof}

In the case that $L$ is replaced by a field $\Lambda$ of characteristic $\ell$ (and $\mathbb{L}$ an irreducible $W_E$-representation over $\Lambda$) we proved the stronger statement that $\mathcal{F}$ is automatically cuspidal using averaging functors in \Cref{sec:aver-funct-via-2-hecke-eigensheaves-of-irred-autom-cuspidal}. We don't know if this stronger statement is true in general.

\begin{lemma}
  \label{sec:farg-conj-gl_n-1-hecke-eigensheaf-with-irreducible-stalk}
  With the same notation from before, assume moreover that for some $b\in B(G)$ basic,
  the stalk $\mathcal{F}_b:=j^\ast_b\mathcal{F}$ corresponds to an
  irreducible representation. Then for each $c\in B(G)$ basic, the stalk $\mathcal{F}_c$ is irreducible.
\end{lemma}
\begin{proof}
  Take $b,c\in B(G)$ basic with
  $\mathrm{deg}(\mathcal{E}_c)=\mathrm{deg}(\mathcal{E}_b)+1$.  We
  already know that $\mathcal{F}$ is supported on the semistable locus
  and that $\mathcal{F}_c$ is a complex of admissible supercuspidal representations.
  Then
  \[
    T_{V_{\rm std}^\ast}(\mathcal{F}_c)\cong \mathbb{L}^\vee\otimes_{L}
    \mathcal{F}_b
  \]
  as $W_{E}$-equivariant sheaves. The right hand side is irreducible (as a
  $W_{E}$-equivariant complex of sheaves). Thus, $T_{V_{\rm std}^\ast}$ kills all
  irreducible Jordan-H\"older factors of cohomology sheaves of $\mathcal{F}_c$, except one.  But using our assumption that $L$ is of characteristic $0$
  \[
    \mathcal{F}_c=T_{L}(\mathcal{F}_c)\to
    T_{V_{\rm std}}(T_{V_{\rm std}^\ast}(\mathcal{F}_c))
  \]
  is a split injection, and thus $\mathcal{F}_c$ is an irreducible representation, placed in some degree. 
\end{proof}

\subsection{The spectral action and the definition of
  $\mathrm{Aut}_{{\mathbb{L}}}$}
\label{sec:spectral-action-and-definition-of-aut-E}

We continue to use the notation of the previous section.
Let
\[
  i_{\mathbb{L}}\colon \Spec(L)\to X_{\hat{G}}
\]
be the natural morphism with image the point determined by $\mathbb{L}$. As the $W_E$-representation $\mathbb{L}$ is assumed to be irreducible, the image of $i_{\mathbb{L}}$ lands inside the smooth locus of $X_{\hat{G}}$, and thus the object
\[
  k(\mathbb{L})_{\mathrm{reg}}:=i_{\mathbb{L},\ast}(L)
\]
lies in $\mathrm{IndPerf}(X_{\hat{G}})$. More concretely, the residual gerbe of $X_{\hat{G}}$ at the image of $i_{\mathbb{L}}$ is given by
\[
  [\Spec(L)/\mathbb{G}_{m,L}]
\]
and the (underived) stalk of $k(\mathbb{L})_{\mathrm{reg}}$ at $x$ is the regular representation
\[
  \mathcal{O}(\mathbb{G}_m)
\]
of $\mathbb{G}_m$. Let us write
\[
  k(\mathbb{L})_{\mathrm{reg}}=\bigoplus\limits_{i\in \Z}k(\mathbb{L})_i,
\]
with $k(\mathbb{L})_i\subseteq k(\mathbb{L})_{\mathrm{reg}}$ the summand where $\Gm$ acts via $z\mapsto z^i$.

Let $\psi\colon E\to L^\times$ be a non-trivial character.
This yields the standard Whittaker datum $(B,\psi)$ for $G=\GL_n$, i.e., $B$ is the standard upper triangular Borel, and (using abuse of notation) $\psi$ denotes the regular character
\[
  \psi\colon U(E)\twoheadrightarrow U/[U,U](E)\cong E^{n-1}\xrightarrow{\sum }E\xrightarrow{\psi}L^\times
\]
with $U\subseteq B$ the unipotent radial of $B$, and $\sum$ the summation.

This yields the so-called \textit{Whittaker sheaf}
\[
  \mathcal{W}_\psi:=j_{1,!}(\mathcal{F}_{\mathrm{cInd}^{G_1(E)}_{U(E)} \psi})\in D_{\lis}(\Bun_n,L).
\]

We can now make the following definition. {Recall the spectral action of \Cref{sec:work-of-fargues-scholze}, denoted with $\ast$.}

\begin{definition}
  \label{definition-of-aut-E}
  For $\mathbb{L}$ an irreducible $W_E$-representation (as before) set
  \[
    \mathrm{Aut}_{\mathbb{L}}:=k(\mathbb{L})_{\mathrm{reg}}\ast \mathcal{W}_\psi \in
    D_\lis(\Bun_n,L).
  \]
\end{definition}

We note that the symbol $\mathrm{Aut}_{{\mathbb{L}}}$ is for now just a notation; the
rest of this text is devoted to showing that this object of
$D_\lis(\Bun_n,L)$ indeed has the properties prescribed by
Fargues' conjecture and therefore deserves its name.

The Hecke eigensheaf property of $\mathrm{Aut}_{\mathbb{L}}$ is formal, and holds more generally as recorded in the following
proposition.

\begin{proposition}
  \label{hecke-eigensheaf-property-for-k(E)-star-anything}
  Let $\mathcal{G}\in D_{\lis}(\Bun_n,L)$ be arbitrary, and let $\mathbb{L}$ be as
  before. Then $k(\mathbb{L})_{\mathrm{reg}} \ast \mathcal{G}$ is an Hecke eigensheaf
  (possibly zero) with eigenvalue $\mathbb{L}$.
\end{proposition}
\begin{proof}
  Let $V\in \mathrm{Rep}_L(\hat{G})$. Then there is a natural isomorphism
  \[
    \begin{matrix}
      T_{V}(k(\mathbb{L})_{\mathrm{reg}}\ast \mathcal{F}) & \cong & f^\ast(V)\ast k(\mathbb{L})_{\mathrm{reg}}\ast \mathcal{F} \\
      & \cong & (f^\ast(V)\otimes_{\mathcal{O}_{X_{\hat{G}}}} k(\mathbb{L})_{\mathrm{reg}})\ast \mathcal{F} \\
      & \cong & (r_{V}(\mathbb{L})\otimes_{L} k(\mathbb{L})_{\mathrm{reg}})\ast \mathcal{F} \\
      & \cong & r_V(\mathbb{L})\otimes_{L} k(\mathbb{L})_{\mathrm{reg}}\ast \mathcal{F}
    \end{matrix}
  \]
  (compatible with the $W_{E}$-action), and similarly for every finite
  set $I$, $V_i\in \mathrm{Rep}(\hat{G}), i\in I$.
\end{proof}

This proposition applies in particular to $\mathcal{G}=\mathcal{W}_{\psi}$, and hence the object $\mathrm{Aut}_{\mathbb{L}}$ is a Hecke eigensheaf as desired. {We also know that it is supported on the semistable locus and given by supercuspidal representations, by \Cref{sec:gener-hecke-eigensh-1-hecke-eigensheaf-supported-on-semistable-locus}. However, it could potentially be zero.} That this is not the case will be deduced in the next subsection from the local Langlands correspondence.

\begin{remark}
\label{motivation-from-categorical-langlands}
To motivate the definition of $\mathrm{Aut}_{\mathbb{L}}$, let us recall that Fargues'conjecture admits a strong conjectural enhancement: it is expected that one has an equivalence
$$
\mathrm{IndCoh}(X_{\hat{G}}) \cong D_\lis(\Bun_n,L)
$$
(recall that here $L$ is assumed to be algebraically closed of characteristic $0$), compatibly with the actions of $\mathrm{IndPerf}(X_{\hat{G}})$ on both sides (acting via tensor product on the left), cf.\ \cite[Conjecture IX.1.4.]{fargues2021geometrization}. Via this equivalence, if $\mathbb{L}$ is as above, the Hecke eigensheaf whose existence is predicted by Fargues' conjecture should correspond to $k(\mathbb{L})_{\mathrm{reg}}$ and the Whittaker sheaf should correspond to the structure sheaf $\mathcal{O}_{X_{\hat{G}}}$\footnote{The representation $\mathrm{cInd}^{G_1(E)}_{U(E)} \psi$ is not finitely generated, and thus $\mathcal{W}_\psi$ is not a compact object in $D_\lis(\Bun_n,L)$, but recall that the structure sheaf $\mathcal{O}_{X_{\hat{G}}}$ is seen as an ind-coherent sheaf, cf. \Cref{sec:work-of-fargues-scholze}, so there is no contradiction.}. Since tautologically
$$
{k(\mathbb{L})_{\mathrm{reg}} }\otimes_{\mathcal{O}_{X_{\hat{G}}}} \mathcal{O}_{X_{\hat{G}}} \cong {k(\mathbb{L})_{\mathrm{reg}}},
$$
since the conjectural categorical equivalence should be equivariant for the action of $\mathrm{IndPerf}(X_{\hat{G}})$ and since $k(\mathbb{L}) \in \mathrm{IndPerf}(X_{\hat{G}})$, we see that the definition of $\mathrm{Aut}_{\mathbb{L}}$ is reasonable (and meaningful independently of the conjecture!).
\end{remark}

\subsection{Generalities on the spectral action of $k(\mathbb{L})_i$}
\label{sec:gener-spectr-acti}

We use the notation from before.
This section is devoted to the study of the spectral actions of the objects
\[
  k(\mathbb{L})_i\in \mathrm{Perf}(X_{\hat{G}})
\]
for $i\in \Z$.
Let us set
\[
  F_i(\mathcal{G}):=k(\mathbb{L})_i\ast \mathcal{G}
\]
for $\mathcal{G}\in D_\lis(\Bun_n,L)$.

Recall the object
\[
  \mathcal{A}v_{\mathbb{L}^\prime}=R\Gamma(W_E,\mathbb{L}^\prime\otimes_L f^\ast V_{\mathrm{std}})\in \Perf(X_{\hat{G}})
\]
associated with a $W_E$-representation $\mathbb{L}^\prime$.

\begin{lemma}
  \label{sec:comp-one-stalk-spectral-object-for-e-irreducible}
  We have an equality
  \[
    \mathcal{A}v_{\mathbb{L}^\vee}\cong k(\mathbb{L})_1[-1]\oplus k(\mathbb{L}(\chi_{\mathrm{cyc}}))_1[-2]\in \mathrm{Perf}(X_{\hat{G}}), 
  \]
  where $\chi_{\mathrm{cyc}}$ denotes the cyclotomic character of $W_E$. In particular,
  \[
    \mathrm{Av}_{\mathbb{L}^\vee}(\mathcal{G})=k(\mathbb{L})_1\ast \mathcal{G}[-1] \oplus k(\mathbb{L}(\chi_{\mathrm{cyc}}))_1\ast \mathcal{G}[-2]
  \]
  for each $\mathcal{G}\in D_{\lis}(\Bun_n,L)$.
\end{lemma}
  \begin{proof}
    The first statement follows by calculating (derived) stalks, which identify with $W_E$-cohomology. The second statement follows from \Cref{sec:aver-funct-via-1-averaging-functor-via-spectral-action}.  
  \end{proof}

  In particular, we see that $F_i(-)$ is a direct summand of $\Av_{\mathbb{L}^\vee}(-)[1]$.

\begin{lemma}
  \label{sec:gener-spectr-acti-1-action-of-k-l-i}
  The following hold true for $i\in \Z$:
  \begin{enumerate}
  \item\label{1} If $\mathcal{G}\in D_\lis(\Bun_n,L)$ has support in the connected component $\Bun_n^{\kappa=d}$, then $F_i(\mathcal{G})$ has support on $\Bun_n^{\kappa=d+i}$.
       \item\label{2} For each $\mathcal{G}\in D_\lis(\Bun_n,L)$ the object $F_i(\mathcal{G})$ is supported on the semistable locus and the stalks at basic points are represented by complexes of supercuspidal representations. 
  \item\label{3} If $\mathcal{G}\in D_\lis(\Bun_n,L)$ has support in the non-semistable locus, then $F_i(\mathcal{G})=0$.
  \item\label{4} If $b\in B(G)$ is basic and $\mathcal{G}=j_{b,!}(\mathcal{F}_\pi)$ for some parabolically induced $G_b(E)$-representation $\pi$, then $F_i(\mathcal{G})=0$.
  \end{enumerate}
\end{lemma}
\begin{proof}
  The cohomology $R\Gamma(W_E,-)$ of some $W_E$-equivariant object in $D_\lis(\Bun_n,L)$ with support in $\Bun_n^{\kappa=d}$ has again support in $\Bun_n^{\kappa=d}$. Hence the analog of \eqref{1} with $F_i$ replaced by $\mathrm{Av}_{\mathbb{L}^\vee}$ follows from the same property for the standard Hecke operator $T_{V_{\mathrm{std}}}$, which is clear. Now, \eqref{1} follows by passing to direct summands.
  Point \eqref{2} follows because $F_i(\mathcal{G})$ is a summand of an Hecke eigensheaf with irreducible eigenvalue $\mathbb{L}$ and \Cref{sec:gener-hecke-eigensh-1-hecke-eigensheaf-supported-on-semistable-locus}.

  Let $e\in \mathcal{O}(X_{\hat{G}})$ be the idempotent cutting out the non-irreducible locus, i.e., the non-irreducible locus (a disjoint union of connected components) is the vanishing locus of $e$. By compatibility with parabolic induction, cf.\ \cite[Chapter IX.6.]{fargues2021geometrization}, the idempotent acts trivially on each ULA object with support on the non-semistable locus. But this implies that $e$ acts trivially on each object supported there as for each $b\in B(G)$, which is non-basic, each Bernstein block in $\mathrm{Rep}^\infty_L(G_b(E))$ contains a non-zero admissible representation. This implies \eqref{3} and a similar argument implies \eqref{4}.
\end{proof}

\subsection{Computation of one stalk of $\mathrm{Aut}_{{\mathbb{L}}}$ and conclusion}
\label{sec:computation-stalk-aut-E-at-one-point}

Our aim is to identify the stalks of
\[
  \mathrm{Aut}_{\mathbb{L}}\cong k(\mathbb{L})_{\mathrm{reg}}\ast \mathcal{W}_\psi\cong \bigoplus\limits_{i\in \Z} k(\mathbb{L})_i\ast \mathcal{W}_\psi
\]
at points in the semistable locus.

We recall that we have fixed a completed algebraic closure
$C$ of $E$, giving rise to a (classical) point $\infty\in X_C$ on the
Fargues-Fontaine curve.

Let $b\in B(G)$ be the basic element with $\kappa(b)=1$.
By \cite[Lecture 24]{scholze_weinstein_berkeley_new_version} the space
\[
  \mathcal{M}_{\mathrm{LT},\infty}
\]
of injections of the trivial rank $n$ vector bundle $\mathcal{E}_1$ into $\mathcal{E}_b$ with cokernel supported at $\infty$ is (the diamond associated with) the infinite level Lubin-Tate space. 

For an irreducible $W_E$-representation $\mathbb{L}^\prime$ over $L$ we will denote by
$\mathrm{LL}_1(\mathbb{L}^\prime)$ the local Langlands
correspondent of $\mathbb{L}^\prime$ and by $\mathrm{LL}_b(\mathbb{L}^\prime)$
the local Jacquet-Langlands correspondent of $\mathrm{LL}_1(\mathbb{L}^\prime)$.

\begin{theorem}
  \label{sec:comp-one-stalk-isotypic-component-of-cohom-of-lubin-tate-space}
  Let $\mathbb{L}$ be an irreducible $W_E$-representation over $L$ and $\mathrm{LL}_1(\mathbb{L})\in \mathrm{Rep}^\infty_L G_1(E)$ its local Langlands correspondent. Then
  \[
    {j_b^\ast}T_{V_\mathrm{std}}({j_{1,!}}\mathcal{F}_{\mathrm{LL}_1(\mathbb{L})})\cong \mathbb{L}\otimes_{L} \mathcal{F}_{\mathrm{LL}_b(\mathbb{L})}
  \]
  as $W_E$-equivariant objects in $D_\lis(\Bun_n,L)$.
\end{theorem}
\begin{proof}
This is discussed in the proof of \cite[Theorem IX.7.4.]{fargues2021geometrization} and follows from the known description of the $\ell$-adic cohomology of the Lubin-Tate tower.
\end{proof}

This incarnation of the local Langlands correspondence allows us to calculate one stalk of {$\mathrm{Aut}_{\mathbb{L}}$}.

\begin{theorem}
  \label{sec:comp-one-stalk-stalk-at-kappa-1}
  Let as before $b\in B(G)$ {basic} with $\kappa(b)=1$. Then
  \[
    j_b^\ast \mathrm{Aut}_{\mathbb{L}}\cong \mathcal{F}_{\mathrm{LL}_b(\mathbb{L})}.
  \]
\end{theorem}
\begin{proof}
 We first analyze
\[
  \Av_{\mathbb{L}^\vee}(\mathcal{W}_\psi)[1]\cong  {k(\mathbb{L})_1}\ast \mathcal{W}_\psi \oplus k(\mathbb{L}(\chi_{\mathrm{cyc}}))_1[-1]\ast \mathcal{W}_{\psi}.
 \]
 By \Cref{sec:gener-spectr-acti-1-action-of-k-l-i} we can replace here $\mathcal{W}_\psi$ by its supercuspidal part.
 Let $\mathcal{B}:=\mathcal{B}_{\mathfrak{s}}\subseteq \mathrm{Rep}_L^\infty G_1(E)$ be a supercuspidal block, and let
    \[
      \mathcal{W}_{\psi, \mathfrak{s}}
    \]
    be the Bernstein component of $\mathcal{W}_\psi$ in $\mathcal{B}$.
    Let
    \[
      \mathfrak{z}
    \]
    be the center of the block $\mathcal{B}$. Then it is known (cf.\ \cite{bernstein1984centre}, \cite{bushnell2003generalized}) that
    \[
      \mathfrak{z}\cong \mathrm{End}_{\mathcal{B}}(\mathcal{W}_{\psi,\mathfrak{s}})\cong L[T^{\pm 1}],
    \]
    that $\mathcal{W}_{\psi,\mathfrak{s}}\in \mathcal{B}$ is projective, and that the functor
    \[
     \mathcal{B}\cong \mathrm{Mod}_{\mathfrak{z}},\ \pi\mapsto \mathrm{Hom}_{\mathcal{B}}(\mathcal{W}_{\psi,\mathfrak{s}},\pi).
   \]
   is an equivalence.
   Implicit in these statements we used the existence and uniqueness of Whittaker models. More precisely, for each irreducible, supercuspidal representation $\pi\in \mathcal{B}$ (such a $\pi$ is unique up to some unramified twist) we have
    \[
      \begin{matrix}
        &  \mathrm{Hom}_{\mathcal{B}}(\mathcal{W}_{\psi,\mathfrak{s}},\pi) \\
        \cong & \mathrm{Hom}_{\Rep^\infty_L G_1(E)}(W_{\psi},\pi) \\
        \cong & \mathrm{Hom}_{\Rep^\infty_L G_1(E) } (\pi^\vee,\mathrm{Ind}_{N(E)}^{G_1(E)} \psi) \\
        \cong & L,
      \end{matrix}
    \]
    where $(-)^\vee$ denotes the smooth dual, and the last step uses the uniqueness and existence for irreducible, supercuspidal representations (cf.\ \cite{gelfand_kazhdan_representations}).

    Let
    \[
      \mathcal{C}\subseteq \Rep^\infty_L G_b(E)
    \]
    be a supercuspidal Bernstein block. By \eqref{2} it suffices to describe the compositions
    \[
      F_{\mathcal{B},\mathcal{C}}\colon \mathcal{B}\subseteq D_{\lis}(\Bun_n,L)\xrightarrow{\Av_{\mathbb{L}^\vee}[1]} D_\lis(\Bun_n^{\mathrm{ss}},L)\to D(\mathcal{C})
    \]
    for all such $\mathcal{B}, \mathcal{C}$. The block $\mathcal{C}$ is again equivalent to modules over its center $R$, and thus (by commutation {of $\mathcal{F}_{\mathcal{B},\mathcal{C}}$ with direct sums}) we have
    \[
      F_{\mathcal{B},\mathcal{C}}(-)\cong \mathcal{M}\otimes_{\mathfrak{z}}^{L}(-)
    \]
    for some complex $M\in D(R)\cong D(\mathcal{C})$ with action by $\mathfrak{z}$. As $F_{\mathcal{B},\mathcal{C}}$ commutes with direct products and preserves compact objects, $M$ is a perfect $\mathfrak{z}$-complex. Let $x\in \Spec(\mathfrak{z})$ be a closed point, with corresponding irreducible, supercuspidal representation $\pi^\prime$. Let
    $\mathbb{L}^\prime$ the the Langlands correspondent of $\pi^\prime$.
    Then by \Cref{sec:comp-one-stalk-isotypic-component-of-cohom-of-lubin-tate-space}
    \[
      M\otimes_{\mathfrak{z}}^L k(x)=0,
    \]
    if $\mathbb{L}^\prime$ is not isomorphic to $\mathbb{L}$ or {$\mathbb{L}(\chi_{\mathrm{cyc}})$}. Let $\pi:=\mathrm{LL}_1(\mathbb{L})$ with corresponding closed point $x_0\in \Spec(\mathfrak{z})$, and let $x_1\in \Spec(\mathfrak{z})$ be the point corresponding to {$\mathrm{LL}_1(\mathbb{L}(\chi_{\mathrm{cyc}}))$}.
    Let $y_0, y_1\in \Spec(R)$ be the closed points corresponding to  $\mathrm{LL}_b(\mathbb{L})$ resp. {$\mathrm{LL}_b(\mathbb{L}(\chi_{\mathrm{cyc}}))$}.
    Then we can deduce from \Cref{sec:comp-one-stalk-isotypic-component-of-cohom-of-lubin-tate-space} that
    \[
      M\otimes_{\mathfrak{z}}^L k(x_0)\cong k(y_0)[1]\oplus k(y_0),
    \]
    while
    \[
            M\otimes_{\mathfrak{z}}^L k(x_1)\cong k(y_1)\oplus k(y_1)[-1].
          \]
          From here we can deduce that $M\cong k(x_0)\oplus k(x_1)[-1]$ (as a $\mathfrak{z}$-complex). This implies
          \[
            j^\ast_b\mathrm{Av}_{\mathbb{L}^{\vee}}(\mathcal{W}_{\psi})[1]\cong \mathcal{F}_{\mathrm{LL}_b(\mathbb{L})}\oplus \mathcal{F}_{\mathrm{LL}_b(\mathbb{L}(\chi_{\mathrm{cyc}}))}[-1].
          \]
          This first implies that 
          $$
           k(\mathbb{L})_1\ast \mathcal{W}_{\psi} \neq 0.
           $$
           Indeed, if it were zero, we would get 
     $$
      k(\mathbb{L}(\chi_{\mathrm{cyc}}))_1\ast \mathcal{W}_\psi[-1] \cong \mathcal{F}_{\mathrm{LL}_b(\mathbb{L})}\oplus \mathcal{F}_{\mathrm{LL}_b(\mathbb{L}(\chi_{\mathrm{cyc}}))}[-1].
      $$
      But applying the same reasoning as above to $\mathbb{L}(\chi_{\rm cyc})$ instead of $\mathbb{L}$, we have 
      $$
    {k(\mathbb{L}(\chi_{\rm cyc}))_1}\ast \mathcal{W}_\psi \oplus k(\mathbb{L}(\chi_{\mathrm{cyc}}^2))_1[-1]\ast \mathcal{W}_{\psi} \cong  \mathcal{F}_{\mathrm{LL}_b(\mathbb{L}(\chi_{\rm cyc}))}\oplus \mathcal{F}_{\mathrm{LL}_b(\mathbb{L}(\chi_{\mathrm{cyc}}^2))}[-1] 
      $$      
       contradicting this isomorphism (since $\mathcal{F}_{\mathrm{LL}_b(\mathbb{L})}[1]$ is not a direct summand of the right hand side). Similarly, we must have
       $$
           k(\mathbb{L}(\chi_{\rm cyc}))_1\ast \mathcal{W}_{\psi} \neq 0.
           $$
           Therefore, we either have
          \[
            k(\mathbb{L})_1\ast \mathcal{W}_{\psi}\cong j_{b,!}\mathcal{F}_{\mathrm{LL}_b(\mathbb{L})},
          \]
          or
          \[
            k(\mathbb{L})_1\ast \mathcal{W}_{\psi}\cong j_{b,!}\mathcal{F}_{\mathrm{LL}_b(\mathbb{L}(\chi_{\mathrm{cyc}}))}[-1].
          \]
                   Assume the second. This implies
          \[
            k(\mathbb{L}(\chi_{\mathrm{cyc}}))_1\ast \mathcal{W}_\psi[-1]\cong j_{b,!}\mathcal{F}_{\mathrm{LL}_b(\mathbb{L})},
          \]
          i.e., that
          \[
            \mathrm{Av}_{\mathbb{L}(\chi_{\mathrm{cyc}})^\vee}(\mathcal{W}_\psi)=k(\mathbb{L}(\chi_{\mathrm{cyc}}))_1\ast \mathcal{W}_\psi[-1]\oplus k(\mathbb{L}(\chi_{\mathrm{cyc}}^{2}))_1\ast \mathcal{W}_\psi[-2]
          \]
          has a contribution in degree $0$. But what we have proven so far (with $\mathbb{L}$ replaced by $\mathbb{L}(\chi_{\mathrm{cyc}})$) implies that this is not the case. This finishes the proof.
\end{proof}

By \Cref{sec:farg-conj-gl_n-1-hecke-eigensheaf-with-irreducible-stalk} and \Cref{sec:comp-one-stalk-stalk-at-kappa-1} we can conclude that each basic stalk of $\mathrm{Aut}_{\mathbb{L}}$ is irreducible. We also know it is a single supercuspidal representation, put in some degree. Let us identify this supercuspidal representation. 

We will only explain how to identify the representation of $G_b(E)$ corresponding to $j_b^\ast \mathrm{Aut}_{\mathbb{L}}$ for $b$ basic with $\kappa(b) \geq 1$. The case where $\kappa(b) \leq 1$ is dealt with similarly, by replacing in what follows the standard representation $V_{\rm std}$ of $\hat{G}$ by its dual. 

We argue inductively. \Cref{sec:comp-one-stalk-stalk-at-kappa-1} tells us that 
$$
j_{b}^\ast \mathrm{Aut}_{\mathbb{L}} \cong \mathcal{F}_{\mathrm{LL}_b(\mathbb{L})},
$$
where $b$ is the basic element with $\kappa(b)=1$. 

Let $b\in B(G)$ basic with $\kappa(b)\geq 1$, and $c$ basic with $\kappa(c)=\kappa(b)+1$. Assume we know that $j_b^\ast \mathrm{Aut}_{\mathbb{L}}$ agrees up to shift with $\mathcal{F}_{\mathrm{LL}_b(\mathbb{L})}$:
$$
 j_b^\ast \mathrm{Aut}_{\mathbb{L}}\cong \mathcal{F}_{\mathrm{LL}_b(\mathbb{L})}[k_b],
 $$
 $k_b\in \Z$. 

By the Hecke eigensheaf property of $\mathrm{Aut}_{\mathbb{L}}$ and our assumption, we
 have
 $$
 j_c^*\mathrm{Aut}_{\mathbb{L}} \otimes_L \mathbb{L} \cong j_c^\ast T_{V_{\rm std}}(j_{b,!} \mathcal{F}_{\mathrm{LL}_b(\mathbb{L})}[k_b]).
 $$
 Therefore, by \cite[Theorem 6.5.2]{kaletha_weinstein_on_the_kottwitz_conjecture_for_local_shimura_varieties}, we deduce that
$$
 \mathrm{tr}(j_c^*\mathrm{Aut}_{\mathbb{L}}) = (-1)^{k_b}\mathrm{tr}(\mathrm{LL}_c(\mathbb{L}))
 $$
  in $\mathrm{Dist}(G_c(E)_{\rm ell},L)^{G_c(E)}$, where $G_c(E)_{\rm ell}$ denotes the elliptic regular locus in $G_c(E)$.
 Since $j_c^*\mathrm{Aut}_{\mathbb{L}}$ corresponds to an irreducible supercuspidal representation, and since such representations are entirely determined by the restriction of their Harish-Chandra distribution character to the elliptic locus, this forces $j_c^*\mathrm{Aut}_{\mathbb{L}}$ to agree with $\mathcal{F}_{\mathrm{LL}_c(\mathbb{L})}$, up to some shift $k_c \in \Z$ (having the same parity as $k_b$). This concludes the induction step. 
 
 Finally, it remains to show that all the shifts $k_b$ appearing above are zero, i.e. that $\mathrm{Aut}_E$ is concentrated in degree $0$ (we so far only know that $k_b=0$ if $b$ is such that $\kappa(b)=1$). But concentration in degree $0$ is implied by \cite[Theorem 1.1]{hansensupercuspidal}.
\\

Altogether, this concludes the proof of \Cref{sec:introduction-2-main-result-on-fargues-sheaf}.

\begin{remark}
\label{remark-concentration-in-degree-0}
Let $b, c\in B(G)$ basic with $\kappa(c)=\kappa(b)+1$. We expect that
$$
R\Gamma(\mathcal{M}_{b,c}, L)
$$
is concentrated in degrees $[0,n-1]$. When $b=1$, this can be deduced from the fact that $\mathcal{M}_{b,c}$ is the infinite level Lubin-Tate space, which is an inverse limit of smooth Stein rigid-analytic varieties, but we do not know how to prove it for all $b$. If this were true, using commutation of $T_{V_{\rm std}}$ with Verdier duality, one could deduce concentration in degree $0$ of $j_b^\ast \mathrm{Aut}_{\mathbb{L}}$ for all $b$ (we learnt this argument from a paper of Fargues, \cite{fargues_involution_zelevinsky}; it was recently reconsidered by Hansen, \cite{hansensupercuspidal}).

Note that $\mathcal{M}_{b,c}$ is Zariski-closed in the open perfectoid unit ball $X$ of dimension $n$ parametrizing maps from $\mathcal{E}_b$ to $\mathcal{E}_c$, cut out by one equation. So in fact concentration in degree $0$ of $\mathrm{Aut}_{\mathbb{L}}$ would follow from the statement that if $Z$ is a Zariski-closed subspace in the $n$-dimensional affinoid perfectoid ball, $R\Gamma(Z,L)$ vanishes in degrees bigger than the Krull dimension of $Z$\footnote{When applying this result to $\mathcal{M}_{b,c}$, which lives in the $n$-dimensional perfectoid \textit{open} unit ball, one must write $\mathcal{M}_{b,c}$ as the filtered colimit of its intersections with affinoid perfectoid balls of smaller radius, which may create some $R^1\lim$ contribution in the cohomology. However, it can be get rid of by looking at the sign of the trace distribution.}. We refer the interested reader to \cite{hansensupercuspidal} for more on this.
\end{remark}

\bibliography{biblio.bib} \bibliographystyle{plain}

\end{document}